\documentclass[12pt]{amsart}
\usepackage{color}
\usepackage{graphicx}
\usepackage{amsmath}
\usepackage{amssymb}
\usepackage{amscd}
\usepackage{url}
\usepackage{verbatim} 
\usepackage{color}
\usepackage{stmaryrd}
\usepackage{amsthm}
\usepackage{hyperref}
 
\title{Dynamics on an infinite surface with the lattice property}
\author{W. Patrick Hooper}
\address{Department of Mathematics, Northwestern University\\
2033 Sheridan Road\\
Evanston, IL 60208-2730, USA (phone: 847-491-2853, fax: 847-491-8906)}
\email{wphooper@math.northwestern.edu}
\subjclass[2000]{37D40;37D50,37E99,32G15}

\newtheorem{theorem}{Theorem}
\newtheorem{proposition}[theorem]{Proposition}
\newtheorem{lemma}[theorem]{Lemma}
\newtheorem{remark}[theorem]{Remark}
\newtheorem{corollary}[theorem]{Corollary}

\def\C{\mathbb{C}}%
\def\N{\mathbb{N}}%
\def\P{\mathbb{P}}%
\def\Q{\mathbb{Q}}%
\def\R{\mathbb{R}}%
\def\Z{\mathbb{Z}}%
\def\RP{\mathbb{RP}}%

\def\GL{\textit{GL}}
\def\SL{\textit{SL}}
\def\SO{\textit{SO}}
\def\PGL{\textit{PGL}}

\renewcommand{\hom}[1]{\ensuremath{{\llbracket #1 \rrbracket}}}%
\def\ker{\textit{ker}}%
\newcommand{\nullset}{\emptyset}

\def\Aut{\textit{Aut}}%
\def\del{\partial}%
\def\dev{\textit{dev}}%
\def\H{\mathbb H}%
\def\hol{\textit{hol}}                                 

\def\isomS1{\textit{Isom}_+(S^1)}
\def\rt3{\sqrt{3}}

\def\G{{\mathcal G}}
\def\Cyl{{\mathcal C}}
\renewcommand{\v}[1]{{\bf{#1}}}
\newcommand{\hp}[2]{{\langle \hspace{-3pt} \langle #1, #2 \rangle \hspace{-3pt} \rangle}}
\def\HH{{\bar{H}}}
\def\Poly{{\bf P}}
\def\KH{{\mathbb K\mathbb H}}

\makeatletter
\def\@strippedMR{}
\def\@scanforMR#1#2#3\endscan{%
  \ifx#1M\ifx#2R\def\@strippedMR{#3}%
  \else\def\@strippedMR{#1#2#3}%
  \fi\fi}
\renewcommand\MR[1]{\relax\ifhmode\unskip\spacefactor3000 \space\fi
  \@scanforMR#1\endscan
  MR\MRhref{\@strippedMR}{\@strippedMR}}
\makeatother

\begin{document}
\begin{abstract}
Dynamical systems on an infinite translation surface with the lattice property are studied. 
The geodesic flow on this surface is found to be recurrent in all but countably many rational directions.
Hyperbolic elements of the affine automorphism group are found to be nonrecurrent, and 
a precise formula regarding their action on cylinders is proven. A deformation of the surface
in the space of translation surfaces is found, which ``behaves nicely'' with the
geodesic flow and action of the affine automorphism group.
\end{abstract}
\maketitle

In this paper, we study dynamical systems on an infinite translation surface with the lattice property. 
We build the surface $S_1$ by gluing together two polygonal parabolas as in figure \ref{fig:s1}.
This surface is infinite in many respects; it has
infinite genus, infinite area, and two cone points with infinite cone angle.

This study is motivated by the study of closed, finite area translation surfaces with the lattice property,
a property which has been found to have great significance for dynamics. 
For instance, Veech has shown the geodesic flow on 
a finite area translation surface with the lattice property satisfies a dichotomy. In every direction, the geodesic flow
is either completely periodic (decomposes into periodic cylinders), or minimal and uniquely ergodic. See \cite{V} or \cite{MT}.

The surface $S_1$ may be obtained from some of Veech's original examples by a limiting process. See 
section \ref{sect:family}.
Moreover, theorem \ref{thm:veech_groups} demonstrates that $S_1$ has the lattice property. The geodesic flow on the surface supports a 
trichotomy. In directions of rational slope, the flow is either completely periodic or highly nonrecurrent, decomposing into strips. 
In irrational directions, the flow is recurrent, but the flow contains no closed trajectories and no saddle connections. See 
section \ref{sect:recurrence}.

\begin{figure}[h]
\begin{center}
\includegraphics[width=3in]{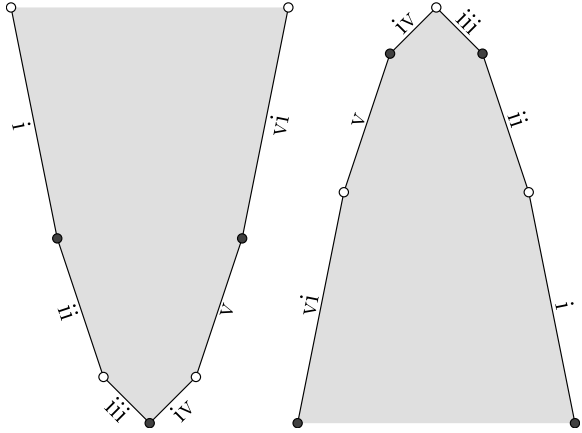}
\caption{The translation surface $S_1$ is built from two infinite polygons, the convex hull of the points 
$(n,n^2)$ for $n \in \Z$ and the convex hull of the points $(n,-n^2)$. Roman numerals indicate edges identified by translations.}
\label{fig:s1}
\end{center}
\end{figure}

We study the dynamics of the action of the affine automorphism group in section \ref{sect:G_dynamics}.
Our theorem governs the dynamics of hyperbolic elements, $\hat{H}$, of the affine
automorphism group of $S_1$. 
The action of $\hat{H}$ on $S_1$ is not recurrent. Nonetheless, cylinders are distributed evenly,
in the following sense. Theorem \ref{thm:asymptotics_of_cylinder_intersections} states
there is a positive constant
$\kappa_H$ depending only on $\hat{H}$ so that given any two cylinders 
${\mathcal A}$ and ${\mathcal B}$ on $S_1$, 
\begin{equation}
\label{eq:asymptotic_formula}
\lim_{m \to \infty} m^{\frac{3}{2}} \textit{Area}\big(\hat{H}^m({\mathcal A}) \cap {\mathcal B}\big)=
\kappa_H \textit{Area}({\mathcal A}) \textit{Area}({\mathcal B}).
\end{equation}
The theorem relies on the understanding of the action of the affine automorphism 
group on homology built up in section \ref{sect:G_dynamics}. The constant $\kappa_H$ has a precise definition,
related to billiards in hyperbolic triangles, and can be 
explicitly computed for any given $\hat{H}$. See theorem \ref{thm:asymptotics_of_homology} 
and appendix \ref{sect:appendix_trace}.

The pair consisting of our surface and its affine automorphism group,
$(S_1, \Aut(S_1))$, exhibits a canonical deformation along a path $(S_c, \Aut(S_c))$ parameterized by real numbers $c \geq 1$. 
Figure \ref{fig:surface_cylinders} shows a surface $S_c$ with $c>1$.
The translation surfaces
$S_c$ are all canonically homeomorphic to $S_1$. 
The groups $\Aut(S_c)$ are canonically isomorphic to $\Aut(S_1)$. The action of $\Aut(S_c)$
on $S_c$ is the same as the action of $\Aut(S_1)$ on $S_1$, up to
conjugacy by the canonical homeomorphism and isotopy.
See section \ref{sect:automorphisms}.

The deformation of $S_1$ preserves a natural coding of geodesic trajectories as well.  
We discover that the surfaces $S_c$ for $c \geq 1$ each have the same saddle connections. That is, up to
homotopy, the canonical homeomorphisms between the surfaces send saddle connections to saddle connections. 
See theorem \ref{thm:classification_of_saddle_connections}.  One way to code trajectories
of the geodesic flow is to triangulate the surface by cutting along disjoint 
saddle connections. A trajectory is coded by the sequence of these saddle connections it crosses.
Let $c_1,c_2 \geq 1$. Since $S_{c_1}$ and $S_{c_2}$ 
have the same saddle connections, it makes sense to compare the coding of trajectories.
In section \ref{sect:saddles}, we show that every coding of a trajectory in $S_{c_1}$ also appears as a coding for a trajectory in $S_{c_2}$.

A brief background on translation surfaces is provided in the next section. In section \ref{sect:family}, we construct our
infinite surfaces. The introduction to section \ref{sect:automorphisms} 
states theorems describing the affine automorphism and Veech groups
of our surfaces.

The remainder of the paper may be divided into three paths, each dependent only on the 
descriptions of the affine automorphism and Veech groups. The reader is invited
to take the path of her choice. Continuing along section \ref{sect:automorphisms} will result in
proofs that affine automorphism and Veech groups are as described in the theorems, as well as
proofs that the surfaces have the same saddle connections and codings of trajectories. 
A walk through section \ref{sect:recurrence} will prove that the geodesic flow in irrational directions on the
surface $S_1$ is recurrent. A stroll through section \ref{sect:G_dynamics} will uncover a detailed understanding of
the action of the affine automorphism groups on homology, which is used to prove the asymptotic formula given in 
equation \ref{eq:asymptotic_formula} above. Proofs of facts in section \ref{sect:G_dynamics} that
require meandering have been moved to the appendix.

The author would like to thank Matt Bainbridge, François Ledrappier, Rich Schwartz, Yaroslav Vorobets,
Barak Weiss, and Amie Wilkinson for helpful discussions
involving this work.

\section{Translation surfaces}

Translation surfaces can be defined in many different ways. We choose the point of view of polygons in $\R^2$ up to 
translation. If $P$ and $Q$ are polygons in $\R^2$, we say  $P$ and $Q$ are equivalent, $P \sim Q$, if there is a translation taking the vertices of $P$ 
to the vertices of $Q$ and preserving the counterclockwise cyclic ordering. We will always use the counterclockwise orientation on the edges of our polygons.

A {\em translation surface} is a surface built by identifying edges of polygons. Our edge identifications must be made by translations. Two edges may be identified only if they are parallel, have the same length, and opposite orientations. 
A simple example is the square torus, the surface obtained from a square by identifying opposite edges by translations.
If our surface is built from finitely many polygons, we insist that it be closed, and when built from countably many polygons,
we insist that it be complete (every Cauchy sequence must converge). In general, the surface is not locally isometric to 
$\R^2$. At places where the vertices of polygons are 
identified there may be cone singularities. When built from finitely many polygons, a translation
surface's cone angles must be integer multiples of $2 \pi$. In the infinite case,
there may be infinite cone angles. 
We use $\Sigma$ to denote the set of all cone singularities.

There are always many ways to build the same translation surface from polygons. We say two translation surfaces $S$ and
$S'$ are {\em the same} if they may be cut into equivalent polygons which are identified by translations 
in the same combinatorial way. This is equivalent to the existence of a direction preserving isometry from $S$ to $S'$. Here 
we use the metric and notion of direction induced by regarding our polygons as subsets of $\R^2$. Because translation surfaces are built by gluing together pieces of $\R^2$ by translations, which preserve distance and direction, 
both notions are natural on translation surfaces. The notion of direction may be viewed as a 
geodesic flow invariant fibration 
of the unit tangent bundle of the surface $S$ over the circle, which restricts to a trivial bundle
on $S \smallsetminus \Sigma$.

There is a natural action of the affine group on translation surfaces. We 
consider the group $\GL(2, \R)$ acting on $\R^2$ in the usual way, by matrix multiplication on column vectors. If 
$S$ is obtained by gluing polygons $P_1, P_2, \ldots$ together by translations determined by edge identifications, then
for $A \in \GL(2, \R)$ the surface $A(S)$ is obtained by gluing $A(P_1), A(P_2), \ldots$ by the same edge identifications. 
This operation constructs a new translation surface, because affine transformations send parallel lines to parallel lines, preserve
ratios of lengths of parallel segments, and uniformly preserve or reverse orientations. 
In a possible abuse of notation, there is a natural map
$S \to A(S)$ induced by the maps $A:P_i \to A(P_i)$, which we also denote by $A:S \to A(S)$.

The {\em Veech group} of a translation surface $S$ is the subgroup of 
$\Gamma(S) \subset \GL(2, \R)$ so that if $A \in \Gamma(S)$ then
$A(S)$ and $S$ are the same, in the sense above. This tells us that there exists at least one direction preserving isometry
from $\varphi_A: A(S) \to S$. The collection of  all maps $\hat{A}=\varphi_A \circ A:S \to S$ with $A \in \Gamma(S)$ 
is a group known as the {\em affine automorphism group} of $S$, $\Aut(S)$. The map
${\bf D}:\Aut(S) \to \Gamma(S)$ given by ${\bf D}:\hat{A} \mapsto A$ is called the derivative map
because $\hat{A}$ acts on the tangent plane of a non-singular point of $S$ as the action of $A$ on $\R^2$.

When $S$ is finite area, elements of the Veech group $\Gamma(S)$, must be area preserving and hence 
$\Gamma(S) \subset \SL^\pm(2,\R)$. Here we use $\SL^\pm(2,\R)$ to denote the space of all $2 \times 2$ matrices
with determinant $\pm 1$. So long that $S$ is built from a finite number of polygons, 
Veech showed that $\Gamma(S)$ is a discrete subgroup of $\SL^\pm(2,\R)$.
A surface is said to have the {\em lattice property} if $\Gamma(S)$ is a lattice in $\SL^\pm(2,\R)$. 

The surfaces we consider in this paper have infinite area. Thus, it is possible that there are elements of the Veech group
with determinants other than $\pm 1$. It is also possible that the Veech group is indiscrete. 
In the infinite cases we consider, we will prove that neither of these occur.

\section{A family of translation surfaces}
\label{sect:family}

Veech considered translation surfaces built from regular $n$-gons \cite{V}. Indeed, every surface
built from finitely many regular $n$-gons has the lattice property.

Here is a dynamical way to write a $n$-gon. Take the rotation of the plane given by the matrix
$$R_t=\left[ \begin{array}{cc}
\cos t & - \sin t \\
\sin t & \cos t
\end{array} \right].$$
The regular $n$-gon is the convex hull of the orbit of the point $(1,0)$ under the group generated by $R_{\frac{2 \pi}{n}}$. 

In order to take a limit we conjugate this rotation by the affine transform 
$S_t:(x,y)\mapsto(\frac{y}{\sin t}, \frac{x-1}{\cos t - 1} )$. The purpose of
$S_t$ is to normalize the first three vertices of the polygons. We have
$$S_t(1,0)=(0,0), ~~
S_t(\cos t, \sin t)=(1,1), ~~ \textrm{and} ~~
S_t(\cos t, -\sin t)=(-1,1).$$
Setting $c=\cos t$ and defining $T_{\cos t}=S_t \circ R_t \circ S_t^{-1}$ yields
\begin{equation}
\label{eq:generalized_rotation}
T_c:(x,y) \mapsto \big(c x+(c-1) y+1, (c+1)x+cy+1\big)
\end{equation}
Let $Q_c^+$ be the convex hull of the set of points $\{P_{c}^k=T_c^k (0,0)\}_{k \in \Z}$.
For $c=\cos \frac{2 \pi}{n}$, $Q_c^+$ is an affinely regular $n$-gon. 
For $c=1$ the collection of forward and backward orbits of $(0,0)$ is the set of points 
$\{(n,n^2) ~|~n \in \Z\}$, the integer points on the parabola
$y=x^2$. Finally for $c>1$,
the orbit of $(0,0)$ lies on a hyperbola. Assume $c=\cosh t$. 
Up to an element of the affine group, the orbit of $(0,0)$ is $\{ (\cosh nt, \sinh nt, 1) ~|~n \in \Z\}$.

We will use $Q_c^+$ to build our translation surfaces.
Let $Q_c^-$ be the image of $Q_c^+$ under a rotation
by $\pi$. Now identify each edge of $Q_c^+$ to its image edge in $Q_c^-$ by {\em parallel translation}. 
We call the resulting translation surface $S_c$.
See figure \ref{fig:veechsurfaces} for some of the cases with $c < 1$. The case $S_1$
is drawn in figure \ref{fig:s1}. 

\begin{figure}[h]
\begin{center}
\includegraphics[width=4in]{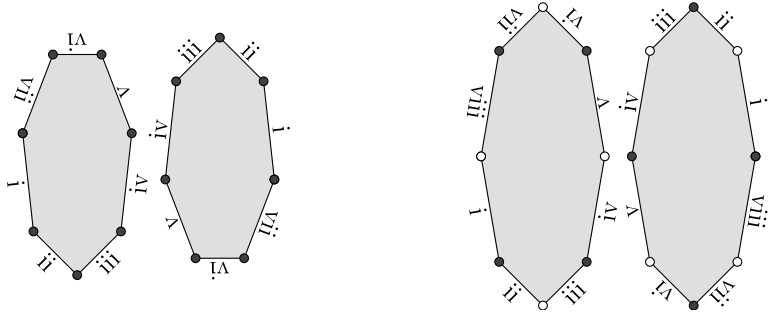}
\caption{The translation surface $S_{\cos \frac{2\pi}{7}}$ and $S_{\cos \frac{\pi}{4}}$ are built from pairs
of affinely regular polygons.}
\label{fig:veechsurfaces}
\end{center}
\end{figure}

In the cases $c \geq 1$, the surfaces $S_c$ are all homeomorphic. 
These surfaces have infinite genus and a pair of singular points with infinite cone angles.
The homeomorphism from $S_c$ to $S_1$ is induced by a homeomorphism $Q_c^+ \to Q_1^+$
which fixes the origin $(0,0)$ and sends each edge of $Q_c^+$ to the corresponding edge of 
$Q_1^+$, preserving ratios of distances on each edges. Conjugating by a rotation by $\pi$ induces a similar map
$Q_c^- \to Q_1^-$. Moreover, because these maps were assumed to preserve ratios of distances
on each edge, the map extends to a homeomorphism $H_c:S_c \to S_1$.


\section{Affine automorphism groups}
\label{sect:automorphisms}

In this section, we compute the affine automorphism groups of the surfaces $S_c$ defined in the previous section.
We cover the case $c=\cos \frac{2 \pi}{n}$ without proof, since the groups are well known.

First, we will describe the Veech groups $\Gamma(S_c)$ for $c=\cos \frac{2 \pi}{n}$ with $n \in \Z$. These
groups always contain $-I_c=-I$ together with the following involutions.
$$A_c=\left[\begin{array}{cc}
-1 & 0 \\
0 & 1
\end{array}\right],
\quad
B_c=
\left[\begin{array}{cc}
-1 & 2 \\
0 & 1
\end{array}\right],
\quad \textrm{and} \quad
C_c=
\left[\begin{array}{cc}
-c & c-1 \\
-c-1 & c
\end{array}\right].
$$
These elements generate a reflection group in a triangle with two ideal vertices and one
vertex with an angle of $\frac{2 \pi}{n}$. In the case of $n$ even, these elements generate 
the Veech group. When $n$ is odd, we must include an additional involution, which folds the
triangle in half.
$$\left[\begin{array}{cc}
-\cos \frac{\pi}{n} & -\frac{\sin^2 \frac{\pi}{n}}{\cos \frac{\pi}{n}} \\
-\cos \frac{\pi}{n} & \cos \frac{\pi}{n}
\end{array}\right]$$
When $n$ is odd, these elements generate a reflection group in a hyperbolic triangle with one
ideal vertex, one right angle, and one angle of $\frac{\pi}{n}$.  These groups are all lattices,
thus these surfaces have the lattice property. 

In our surfaces $S_c$ with $c \geq 1$, the last automorphism makes no appearance.
The Veech groups $\Gamma(S_c)$ for $c=\cos \frac{2 \pi}{n}$ or $c \geq 1$ 
may be thought of as representations of the group 
$$\G^\pm=(\Z_2 \ast \Z_2 \ast \Z_2) \oplus \Z_2=\langle A,B,C,-I ~|~ A^2=B^2=C^2=I \rangle,$$
where $I$ denotes the identity matrix. Given an element $G \in \G^\pm$, 
we will use $G_c$ to denote the corresponding element of $\Gamma(S_c)$.

\begin{theorem}[Veech groups]
\label{thm:veech_groups}
The Veech groups $\Gamma(S_c)\subset \GL(2,\R)$ for $c \geq 1$ 
are generated by the involutions $-I_c=-I$,
$$A_c=\left[\begin{array}{cc}
-1 & 0 \\
0 & 1
\end{array}\right],
\quad
B_c=
\left[\begin{array}{cc}
-1 & 2 \\
0 & 1
\end{array}\right],
\quad \textrm{and} \quad
C_c=
\left[\begin{array}{cc}
-c & c-1 \\
-c-1 & c
\end{array}\right].
$$
\end{theorem}
We will use $\G^\pm_c$ to denote the subgroup $\langle -I_c, A_c, B_c, C_c \rangle \subset \SL(2,\R)$. The theorem 
claims $\G^\pm_c=\Gamma(S_c)$ for $c \geq 1$.

These groups $\G^\pm_c$ are best understood by looking at their action on the hyperbolic plane 
$\H^2=\SL(2,\R)/\SO(2,\R)$.
The elements $A_c$, $B_c$, and $C_c$ act on $\H^2$ as reflections in geodesics.
When $c=1$, the triangle formed
has three ideal vertices, and the Veech group $\Gamma(S_1)$ is precisely the congruence two subgroup of 
$\SL^\pm(2,\Z)$. This group is a lattice. Finally, when $c>1$, the triangle formed has two ideal vertices and one
ultra-ideal vertex. See figure \ref{fig:veechgroup}.

\begin{figure}[h]
\begin{center}
\includegraphics[width=4in]{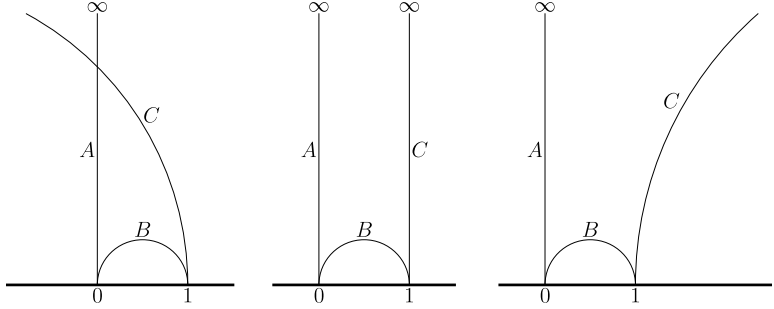}
\caption{This figure shows the geodesics in the upper half plane model of $\H^2$ 
that $A_c$, $B_c$, and $C_c$ reflect in for $c=\cos \frac{\pi}{4}$, $c=1$, and $c=\frac{5}{4}$ from left to right.}
\label{fig:veechgroup}
\end{center}
\end{figure}

We also wish to give a description of the affine automorphism groups $\Aut(S_c)$. Before we do this, it is useful to
work with alternate generators for $\G^\pm$. Define the elements $D=B A$, $E=(-I) C B$. The elements $\{A,D,E,-I\}$ also are
generators for $\G^\pm$. From the above, we have
\begin{equation}
\label{eq:affop}
D_c=\left[\begin{array}{cc}
1 & 2 \\
0 & 1
\end{array}\right]
\quad
E_c=\left[\begin{array}{cc}
-c & c+1 \\
-c-1 & c+2
\end{array}\right]
\end{equation}

The affine automorphism group of a surface $S$ may be generated by choosing for each generator $G$ of 
$\Gamma(S)$ an element $\hat{G} \in \Aut(S)$ so that ${\bf D}(\hat{G})=G$, and by including
generators for the kernel of ${\bf D}:\Aut(S) \to \Gamma(S)$.

We will now describe the affine automorphism groups of $S_c$ for $c=\cos \frac{2\pi}{n}$. The surfaces
$S_c$ are affinely equivalent to gluings of a pair of regular polygons $n$-gons. The surfaces
built from a pair of regular $n$-gons have several affine automorphisms whose derivatives are Euclidean rotations
and reflections. The affine automorphism groups include conjugates of these Euclidean elements.
The groups also contain the automorphism $\hat{D}_c$, which applies a single right Dehn twist to each
horizontal cylinder, and $\hat{E}_c$, which acts by a single right Dehn twist in every slope one cylinder.
When $n$ is odd, these elements generate the affine automorphism group
and $\ker ({\bf D})$ is trivial. For $n$ even,
we also need to include the affine automorphism which swaps the two polygons by translation.
This involution generates $\ker({\bf D})$ when $n$ is even.

The surfaces $S_c$ for $c \geq 1$ also admit cylinder decompositions into horizontal and vertical 
cylinders. See figure \ref{fig:surface_cylinders}.

\begin{theorem}[Affine Automorphisms]
\label{thm:affine_automorphisms}
The subgroup $\ker ({\bf D}) \subset \Aut(S_c)$ is trivial for $c \geq 1$.
Generators for $\Aut(S_c)$ with $c \geq 1$ may be described as follows.
\begin{itemize}
\item $\widehat{-I}_c$ swaps the two pieces of $S_c$, rotating each piece by $\pi$.
\item $\hat{A}_c$ is the automorphism induced by the Euclidean reflection in the vertical line $x=0$, which
preserves the pieces $Q_c^+$ and $Q_c^-$ of $S_c$ and preserves the gluing relations.
\item $\hat{D}_c$ preserves the decomposition of $S_c$ into maximal horizontal cylinders, and acts as a single right Dehn twist in each cylinder.
\item $\hat{E}_c$ preserves the decomposition of $S_c$ into maximal cylinders of slope $1$, and acts as a single right Dehn twist in each cylinder.
\end{itemize}
\end{theorem}
Let $\hat{\G}_c$ denote the group of self-homeomorphisms up to isotopy 
generated by $\langle \widehat{-I}_c, \hat{A}_c, \hat{D}_c, \hat{E}_c \rangle$.
The theorem states that these generators may be realized by affine automorphisms and generate $\Aut(S_c)$.
Both these theorems will be proved in the subsequent subsections.

Recall, the translation surfaces $S_c$ for $c \geq 1$ are canonically homeomorphic. Further, by the above theorem, the action of the affine automorphism groups $\Aut(S_c)$ on $S_c$ are topologically the same. That is,
they are the same up to conjugacy by this homeomorphism and isotopies fixing the singularities. We use this point
of view in the proofs. 

\begin{figure}[h]
\begin{center}
\includegraphics[width=4in]{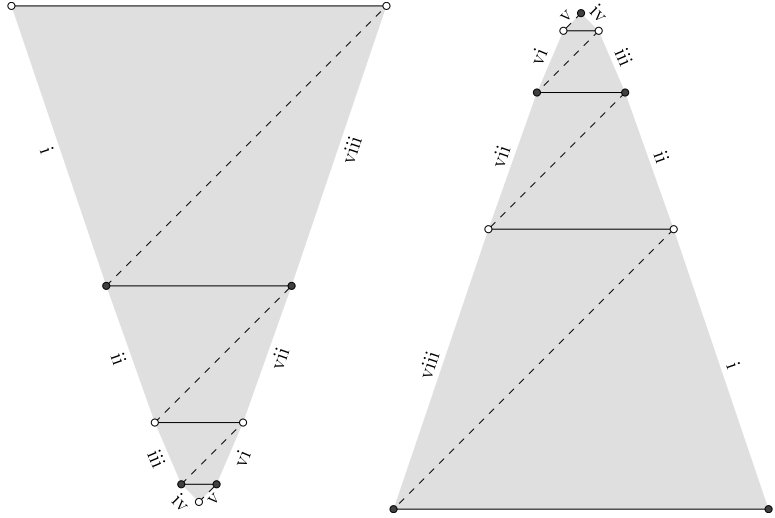}
\caption{The surface $S_c$ with $c=\frac{5}{4}$ is shown. Horizontal and slope one saddle connections
are drawn.}
\label{fig:surface_cylinders}
\end{center}
\end{figure}

\begin{remark}
The automorphism $\hat{F}_c$ corresponding to the element
$$F_c=C_c A_c=\left[ \begin{array}{cc}
c & c -1 \\
c+1 & c
\end{array}\right]$$ 
may be of special interest. The action of $\hat{F}_c$ preserves the decomposition into two pieces,
$Q_c^+$ and $Q_c^-$.
It acts on the top piece as the element of $T_c$ acts on the plane. (See equation \ref{eq:generalized_rotation}). 
When $c \geq 1$, the surface $S_c$ decomposes into a countable number of maximal 
strips in each eigendirection of $F_c$. The action of $\hat{F}_c$ preserves this decomposition into strips.
We number each strip by integers, so that each strip numbered by $n$ is adjacent to the 
strips with numbers $n \pm 1$. There are two possible numberings satisfying this condition. If the correct 
numbering system is chosen, the action of $\hat{F}_c$ sends each strip numbered by $n$ to the strip numbered 
$n+1$.

The action of $\hat{F}_c$ for $c \geq 1$ is as nonrecurrent as possible. Given any compact set 
$K \subset S_c \smallsetminus \Sigma$, there is an $N$ so that for $n>N$, $\hat{F}_c^n(K) \cap K = \nullset$.
\end{remark}

The affine automorphism theorem implies the Veech group theorem,
up to computations of the derivatives of the affine automorphisms.
In the next subsection, we show that
our list of self-homeomorphisms $\widehat{-I}_c$, $\hat{A}_c$, $\hat{D}_c$, and $\hat{E}_c$
are indeed realized by affine automorphisms. 
It is more difficult to show that every affine automorphism
lies in the generated 
group. This will be demonstrated in the 
subsequent subsections.

\subsection{Generators of the affine automorphism groups}
\label{sect:generators}

In this section we prove that the affine automorphism and Veech groups contain the elements we listed in 
theorems \ref{thm:veech_groups} and \ref{thm:affine_automorphisms} in the
cases of $c \geq 1$.

We begin by stating a well known result that gives a way to detect parabolic elements inside the
Veech group. The idea is that a Dehn twist may be performed in a cylinder by a parabolic. See
figure \ref{fig:dehn}.
The height of a Euclidean cylinder divided by its circumference is called the {\em modulus} of the cylinder.

\begin{proposition}[Veech]
\label{prop:parabolic}
Suppose a translation surface has a decomposition into cylinders $\{C_i\}_{i \in \Lambda}$ 
in a direction $\theta$. Suppose further there is a number $m$ such that for 
every cylinder $C_i$, the modulus of $M_i$ of $C_i$ satisfies $m/M_i \in \Z$.
Then there is an affine automorphism of the
translation surface preserving the direction $\theta$, fixing each point on the boundary of each cylinder, 
and acting as an $m/M_i$ power of single right Dehn twist in each cylinder $C_i$. The derivative this affine 
automorphism is a parabolic fixing direction $\theta$.
\end{proposition}

\begin{figure}[h]
\begin{center}
\includegraphics[width=4.5in]{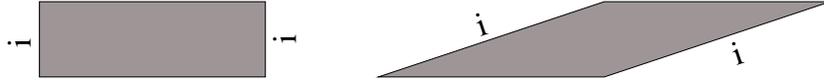}
\caption{A right Dehn twist in a cylinder may be realized by a parabolic (or shear).
The action on the cylinder preserves boundary points, but twists the interior of the cylinder.}
\label{fig:dehn}
\end{center}
\end{figure}

In our cases, each $M_i$ will be equal, hence we get an affine automorphism which acts by a single right 
Dehn twist in each cylinder.

\begin{lemma}
\label{lem:generators}
Let $c \geq 1$.
The subgroup $\G^\pm_c= \langle -I_c, A_c, D_c, E_c \rangle \subset \SL(2,\R)$ 
is contained in the Veech group $\Gamma(S_c)$. 
The self-homeomorphisms of $S_c$ given by
$\widehat{-I}_c$, $\hat{A}_c$, $\hat{D}_c$, and $\hat{E}_c$ 
may be realized by affine automorphisms, and thus generate a subgroup of $\Aut(S_c)$.
\end{lemma}
\begin{proof}
Recall, the surface $S_c$ for $c \geq 1$ was built from two pieces $Q_c^+$ and $Q_c^-$. We defined 
$Q_c^+$ to be the convex hull of the vertices 
$P_i=T_c^i(0,0)$ for $i \in \Z$, with $T_c$ as in equation \ref{eq:generalized_rotation}.
Next $Q_c^-$ was defined to be $Q_c^+$ rotated by $\pi$. $S_c$ is built by gluing the edges of $Q_c^+$
to its image under $Q_c^-$ by parallel translation. Indeed, it is obvious from this definition 
that the rotation by $\pi$ which swaps $Q_c^+$ and $Q_c^-$ restricts to an affine automorphism of the surface, $\widehat{-I}_c \in \Aut(S_c)$. The derivative of $\widehat{-I}_c$ is $-I_c=-I$, which therefore lies in $\Gamma(S_c)$.

Now we will see that the reflection in the line $x=0$ induces an affine automorphism ($\hat{A}$). The reflection
is given the map $r:(x,y)\mapsto(-x,y)$. $Q_c^+$ is preserved because $r(P_i)=P_{-i}$, which follows from the fact
that $r \circ T_c \circ r^{-1} = T_c^{-1}$. The reflection acts in the same way on $Q_c^-$, and thus preserves
gluing relations. Thus, $\hat{A}_c$ is an affine automorphism and its derivative, $A_c$, lies in the Veech group.

We will show that each cylinder in the horizontal cylinder decomposition has the same modulus,
which will prove that $\hat{D}_c$ lies in the affine automorphism group by proposition \ref{prop:parabolic}. Let
$P_i=(x_i,y_i)$. The circumference of the $n-th$ cylinder numbered vertically is given by
$C_n=2 x_{n-1}+2 x_{n}$, and the height is $H_n=y_n-y_{n-1}$. Now let $(x_{n-1},y_{n-1})=(\hat x, \hat y)$, so that
by definition of $T_c$, we have $(x_{n},y_{n})=(c \hat{x}+(c-1)\hat{y}+1,(c+1)\hat{x}+c \hat{y}+1)$. This makes
$$C_n=2(c+1)\hat{x}+2(c-1)\hat{y}+2 \quad \textrm{and} \quad H_n=(c+1)\hat{x}+(c-1)\hat{y}+1.$$
So that the modulus of each cylinder is $\frac{1}{2}$. It can be checked that the parabolic fixing the horizontal direction
and acting as a single right Dehn twist in cylinders of modulus $\frac{1}{2}$ is given by $D_c$.

It is not immediately obvious that there is a decomposition into cylinders in the slope $1$ direction. To see this,
note that there is only one eigendirection corresponding to eigenvalue $-1$ 
of the $\SL(2,\R)$ part of the affine transformation
$$U:(x,y) \mapsto (-cx+(c-1)y+1,-(c+1)x+cy+1)$$
is the slope one direction. It also has the property that $U \circ T_c \circ U^{-1}=T_c^{-1}$, which can be used to show that $U$ swaps $P_i$ with $P_{1-i}$. Therefore segment $\overline{P_{1-i} P_i}$ always has slope one.
The $n$-th slope one 
cylinder is formed by considering the union of trapezoid
obtained by taking the convex hull of the points $P_{n}$, $P_{n+1}$, $P_{1-n}$ and $P_{-n}$ and the same trapezoid rotated by $\pi$ inside $Q_c^-$. Now we will show that the moduli of these cylinders are all equal.
The circumference and height of the $n$-th cylinder in this direction is given below.
$$C_n=\sqrt{2}(x_{n}-x_{1-n}+x_{n+1}-x_{-n})$$
$$H_n=\frac{\sqrt{2}}{2}(x_{n+1}-x_{n},y_{n+1}-y_{n}) \cdot (-1,1)$$
Let $P_n=(\hat{x},\hat{y})$. Then $P_{n+1}=(c \hat{x}+(c-1)\hat{y}+1,(c+1)\hat{x}+c \hat{y}+1)$,
$P_{1-n}=((-c)\hat{x}+(c-1)\hat{y}+1, (-c-1)\hat{x}+c \hat{y}+1)$ and
$P_{-n}=(-\hat{x},\hat{y})$.
We have
$$C_n=\sqrt{2}(2c+2)\hat{x} 
\quad \textrm{and} \quad
H_n=\sqrt{2}\hat{x}.$$
The modulus of each cylinder is $\frac{1}{2c+2}$. Thus by proposition \ref{prop:parabolic},
$\hat{E}_c$ lies in the affine automorphism group. Again, we leave it to the reader to check that the derivative 
of $\hat{E}_c$ must be $E_c$. 
\end{proof}

\subsection{A classification of saddle connections}
\label{sect:saddles}
In this subsection, we will classify the directions in $S_c$ where saddle connections can appear. We begin with $S_1$.

Integers $p$ and $q$ are {\em relatively prime}, if $n, \frac{p}{n}, \frac{q}{n} \in \Z$ implies $n= \pm 1$.
We use the notation $\frac{p}{q} \equiv \frac{r}{s} \pmod{2}$ to say that once the fractions are reduced 
to $\frac{p'}{q'}$ and $\frac{r'}{s'}$ so that numerator
and denominator are relatively prime, we have $p' \equiv r' \pmod{2}$ and $q' \equiv s' \pmod{2}$. 
We use $\frac{p}{q} \not\equiv \frac{r}{s} \pmod{2}$ to denote the negation of this statement.

\begin{proposition}[Saddle connections of $S_1$]
\label{prop:S_1_saddles}
Saddle connections $\sigma \subset S_1$ must have integral holonomy $\hol_1(\sigma)\in \Z^2$. 
A direction contains saddle connections if and only if it has rational slope, $\frac{p}{q}$, 
with $\frac{p}{q} \not \equiv \frac{1}{0} \pmod{2}$.
\end{proposition}
\begin{proof}
The holonomy of a saddle connection must be integral, because the surface $S_1$ was built from two (infinite) polygons
with integer vertices. The action of the Veech group $\Gamma(S_1)$, the congruence two subgroup of 
$\SL^\pm(2,\Z)$ preserves the collection of vectors
$$\textit{RP}=\{(p,q)\in \Z^2\smallsetminus\{(0,0)\}~|~\textrm{$p$ and $q$ are relatively prime}\}.$$
Indeed the orbits of $(0,1)$, $(1,1)$, and $(1,0)$ under $\Gamma(S_1)$ are disjoint and cover $\textit{RP}$. 
Thus, up to the affine automorphism group, the geodesic flow in a direction of rational slope looks like the geodesic flow in the horizontal, slope one, 
or vertical directions. All these directions other than the vertical one contain saddle connections.
Therefore, rational directions contain saddle connections unless they are in the orbit of the vertical direction under 
$\Gamma(S_1)$.
\end{proof}

In order to make a similar statement for $S_c$, we will need to describe the directions that contain saddle connections. 
We will find it useful to note that there is a natural bijective correspondence between directions in the plane modulo rotation by $\pi$, and the boundary of the hyperbolic plane $\del \H^2$. 
This can be seen group theoretically. Directions in the plane
correspond to $S^1= \SL(2,\R)/\textrm{H}$ where 
$$\textrm{H}=\{G \in \SL(2,\R)~|~G\big(\left[\begin{array}{c}1 \\ 0\end{array}\right]\big)= 
\left[\begin{array}{c}\lambda \\ 0\end{array}\right] \textrm{ for some $\lambda > 0$}\}.$$
Both directions mod rotation by $\pi$ and the boundary of the hyperbolic plane correspond to
the real projective line, $\RP^1=\SL(2,\R)/\textrm{H}^\pm$, where 
$$\textrm{H}^\pm=\{G \in \SL(2,\R)~|~G\big(\left[\begin{array}{c}1 \\ 0\end{array}\right]\big)= 
\left[\begin{array}{c}\lambda \\ 0\end{array}\right] \textrm{ for some $\lambda\neq 0$}\}.$$

Now consider the left action of the groups $\G^\pm_c$ on $\RP^1$. There are known natural 
semiconjugacies between the actions of $\G^\pm_c$ for $c>1$ on $\RP^1$ and the action of $\G^\pm_1$.
That is, for all $c > 1$ there is a continuous surjective map $\varphi_c:\RP^1 \to \RP^1$ of degree one so that the following
diagram commutes for all $G \in \G^\pm$
\begin{equation}
\label{varphi_c}
\begin{CD}
\RP^2 @>G_c>> \RP^2 \\
@VV\varphi_cV @VV\varphi_cV 
\\
\RP^2 @>G_1>> \RP^2
\end{CD}
\end{equation}
It will also be important that $\phi_c$ {\em weakly preserves orientation}. That is, if 
$({\bf a},{\bf b},{\bf c}) \in (\RP^1)^3$ 
is a positively oriented triple, then $(\phi_c({\bf a}), \phi_c({\bf b}), \phi_c({\bf c}))$ is not a negatively oriented triple.

We may describe the map $\varphi_c:\RP^1 \to \RP^1$ geometrically as follows. Recall, that for $c>1$, the fundamental domain for the action of $\G^\pm_c$ on $\H^2$ is a triangle $\Delta_c$ with two ideal vertices and one 
ultra-ideal vertex. For $c=1$, the fundamental domain of $\G^\pm_1$ is a triangle $\Delta_1$
with three ideal vertices. See the right
and center sides of figure \ref{fig:veechgroup} respectively. We may choose a homeomorphism between the
closures in $\overline{\H}^2$, $f_c:\overline{\Delta}_c \to \overline{\Delta}_1$
which collapses the ultra-ideal vertex to an ideal vertex, and sends the edges to edges. 
By propagating this homeomorphism around by the groups
$\G^\pm_c$ and $\G^\pm_1$ we can find a unique homeomorphism from 
$\hat{f}_c:\overline{\H}^2 \to \overline{\H}^2$ so that
$G_1 \circ f_c=\hat{f}_c \circ G_c |_{\Delta_c}$ for all $G \in \G^\pm$. The desired map $\varphi_c$ is the restriction of the map 
$\hat{f}_c$ to $\RP^2= \del \H^2$.

The map $\varphi_c$ is reminiscent of the famous devil's staircase, a continuous surjective map 
$[0,1] \to [0,1]$ which contracts intervals in the compliment of a Cantor set to points.
Indeed, the limit set $\Lambda_c$ of the group $\G^\pm_c$ is a $\G^\pm_c$ invariant Cantor set, and
the connected components of the domain of discontinuity, $\RP^1 \smallsetminus \Lambda_c$, 
are contracted to points by $\varphi_c$.

Recall, the surfaces $S_c$ for $c\geq 1$ are canonically homeomorphic up to isotopy fixing the singular points.
We let $h_c:S_c \to S_1$ be such a homeomorphism. We will see that the saddle connections in
$S_c$ and in $S_1$ are topologically the same. We will now make this notion rigorous. We will say two paths $\gamma$ and 
$\gamma'$ in a translation surface are {\em homotopic} (relative to their endpoints) if they have the same endpoints and there is a homotopy from 
$\gamma$ to $\gamma'$ which fixes those endpoints. We do not allow this homotopy to pass through singular points. 
We use $[\gamma]$ to denote the equivalence class of all paths homotopic to $\gamma$.

\begin{theorem}[Classification of saddle connections]
\label{thm:classification_of_saddle_connections}
There is a saddle connection in direction $\theta$ of $S_c$ for $c > 1$ if and only if the direction 
$\varphi_c(\theta)$ contains saddle connections in $S_1$. 
Equivalently, $\theta$ contains saddle connections if and only if $\theta$ is an image of
the horizontal or slope one direction under an element of $\G^\pm_c=\langle -I_c, A_c, D_c, E_c \rangle$. 
Furthermore, the collection of homotopy classes are identical in $S_c$ and $S_1$. That is, for all saddle connections 
$\sigma \subset S_c$ there is a saddle connection $\sigma' \in h_c([\sigma])$ in $S_1$, and for
all saddle connections $\sigma' \subset S_1$ there is a saddle connection $\sigma \in h_c^{-1}([\sigma'])$ in $S_c$.
\end{theorem}

We will prove this theorem by first proving a more abstract lemma. Then we will demonstrate that $S_c$ and $S_1$ satisfy the 
conditions of the lemma.

In the statement of the lemma, we use the concept of the holonomy of a saddle connection. Given any oriented 
path $\gamma:[0,1] \to S$ in a translation  surface which avoids the singularities on $(0,1)$,
there is a canonical development of $\gamma$ into the plane,  $\dev(\gamma):[0,1] \to \R^2$, 
which is a path in the plane well defined up to translation. The holonomy vector
$\hol(\gamma)$ is obtained by subtracting the endpoint of $\dev(\gamma)$ from its starting point. The quantity
$\hol(\gamma)$ is homotopy invariant. The notions of holonomy and the developing map are common in the world of
$(G,X)$ structures; see section $3.4$ of \cite{Thurston}, for instance.

The wedge product between two vectors in $\R^2$ is given by
\begin{equation}
\label{eq:wedge}
(a,b) \wedge (c,d)=ad-bc.
\end{equation}
This is also the signed area of the parallelogram formed by the two vectors.

The function $\textit{sign}:\R \to \{-1,0,1\}$ assigns one to positive numbers, zero to zero, and $-1$ to negative numbers.

\begin{lemma}
\label{lem:same_saddles}
Let $h:S \to T$ be a homeomorphism between translation surfaces satisfying the following statements.
\begin{enumerate}
\item $S$ admits a triangulation by saddle connections.
\item For every saddle connection $\sigma \subset S$ the homotopy class $h([\sigma])$ contains a saddle connection
of $T$.
\item Every pair of saddle connections $\sigma_1, \sigma_2 \subset S$ satisfies
$$\textit{sign} \big( \hol(\sigma_1) \wedge \hol(\sigma_2) \big) = 
\textit{sign} \big( \hol ( h (\sigma_1)) \wedge \hol(h(\sigma_2)) \big).$$
\end{enumerate}
Then, for every saddle connection $\sigma \subset T$, the homotopy class $h^{-1}([\sigma])$ contains a saddle connection
of $S$.
\end{lemma}
\begin{proof}
Let ${\mathcal T}_S$ be the triangulation of $S$ by saddle connections given to us by item 1. 
By item 2, we we can straighten $h({\mathcal T}_S)$ to a triangulation ${\mathcal T}_T$ of $T$ by saddle connections. 

We define the {\em complexity} of a saddle connection $\sigma \subset T$ relative to the triangulation ${\mathcal T}_T$ to be the number of
times $\sigma$ crosses a saddle connection in ${\mathcal T}_T$. 
We assign the saddle connections in ${\mathcal T}_T$ complexity zero.
Supposing the conclusion of the lemma is false, there exists at least one saddle connection $\sigma \subset T$ so
that $h^{-1}([\sigma])$ contains no saddle connection of $S$. We may choose such a saddle connection $\sigma \subset T$ so that it
has minimal complexity with respect to ${\mathcal T}_T$. By the remarks above this minimal complexity must be at least one.

\begin{figure}[h]
\begin{center}
\includegraphics[width=3.5in]{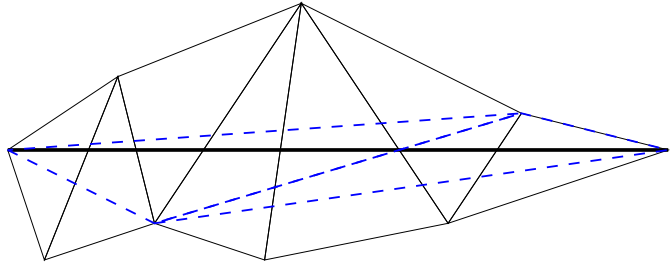}
\caption{The saddle connection $\sigma$ is developed into the plane along with the triangles in ${\mathcal T}_T$ that it intersects.}
\label{fig:saddleconnection}
\end{center}
\end{figure}

Now, we will reduce to a case similar to the complexity one case. We may rotate the surface $T$ so that $\sigma$ is horizontal
and develop the triangles $\sigma$ intersects into the plane as a chain of triangles. See figure \ref{fig:saddleconnection}. Now choose vertices
$v_+$ and $v_-$ below and above $\sigma$ which lie closest to $\sigma$. The convex hull of $\{\sigma, v_+, v_-\}$ is a quadrilateral $Q$ which is
contained in the developed chain of triangles. 
The boundary of $Q$ consists of four saddle connections $\nu_1, \ldots, \nu_4$ with complexity relative to ${\mathcal T}_T$
less than that of $\sigma$. Because we assumed $\sigma$ had minimal complexity, there are saddle connections 
$\nu'_1, \ldots, \nu'_4 \subset S$ in the homotopy classes $h^{-1}([\nu_1]), \ldots, h^{-1}([\nu_4])$ respectively. Finally, because
of item 3, the saddle connections $\nu'_1, \ldots, \nu'_4$ in $S$ must also form a strictly convex quadrilateral $Q'$. 
The quadrilateral $Q'$ must then have diagonals, one of which lies in the homotopy class $h^{-1}([\sigma])$. 
See figure \ref{fig:destroyedsaddle}.
\end{proof}

\begin{figure}[h]
\begin{center}
\includegraphics[width=2in]{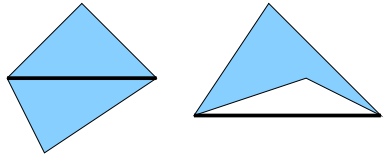}
\caption{To destroy a diagonal of a quadrilateral, the quadrilateral must be made non-convex. This violates property 3 of lemma 
\ref{lem:same_saddles}.}
\label{fig:destroyedsaddle}
\end{center}
\end{figure}

The following proposition implies the classification of saddle connections, theorem \ref{thm:classification_of_saddle_connections}.

\begin{proposition}
\label{prop:lemma_satisfied}
The canonical homeomorphism $h_c^{-1}:S_1 \to S_c$ satisfies the conditions of lemma \ref{lem:same_saddles}.
\end{proposition}
\begin{proof}
Item 1 is trivial. We leave it to the reader to triangulate $S_1$. 

Item 2 follows from proposition \ref{prop:S_1_saddles} together with work done in subsection \ref{sect:generators}. By proposition 
\ref{prop:S_1_saddles} all saddle connections of $S_1$ are the images of saddle connections in the horizontal and slope one directions
under $\G^{\pm}_1$. By work done in subsection \ref{sect:generators}, we know that for each saddle connection $\sigma$ 
in the horizontal and slope one directions that appears in $S_1$, there is a saddle connection $\sigma' \in h_c^{-1}([\sigma])$ in
$S_c$. Furthermore by lemma \ref{lem:generators}, the group action of $\langle \widehat{-I}_c, \hat{A}_c, \hat{D}_c, \hat{E}_c \rangle \subset \Aut(S_c)$ on $S_c$ 
is the same as the group action of 
$\langle \widehat{-I}_1, \hat{A}_1, \hat{D}_1, \hat{E}_1 \rangle \subset \Aut(S_1)$ 
on $S_1$ up to $h_c$ and isotopies preserving the singularities. Therefore,
for every saddle connection $\sigma \subset S_1$, there is a saddle connection $\sigma' \in h_c^{-1}([\sigma])$.

Now we show item 3 holds.
Let $\sigma$ and $\sigma'$ be saddle connections in $S_1$. 
Let $\theta_0$, $\theta_1$ be the horizontal and slope one directions in the circle $S^1$, respectively.
Then we can choose ${\bf v}, {\bf v}' \in \{\theta_0, \theta_1\}$ and $G_1, G'_1 \in \G^\pm_1$ 
such that $G_1({\bf v})=\hol_1(\sigma)$ and $G_1'({\bf v}')=\hol_1(\sigma')$. It follows that
the corresponding elements $G_c, G_c' \in \G^\pm_c$ satisfy
$G_c({\bf v})=\hol_c \circ h_c^{-1}(\sigma)$ and $G_c'({\bf v}')=\hol_c \circ h_c^{-1}(\sigma')$. 
We must prove that 
$$\textit{sign} \big( G_1({\bf v}) \wedge G_1'({\bf v}') \big)=\textit{sign} \big( G_c({\bf v}) \wedge G_c'({\bf v}') \big).$$
This follows essentially from the description of the map $\varphi_c$.
Let $\tilde \varphi_c:S^1 \mapsto S^1$ be a lift of $\varphi_c:\RP^1 \to \RP^1$. 
This lift is unique if we assume the rightward direction is preserved,  $\tilde \varphi_c(1,0)=(1,0)$. 
The map $\tilde \varphi_c$ satisfies the commutative diagram 
 \ref{varphi_c} with $\RP^1$ replaced by $S^1$. Using this diagram we may rewrite the above equation as
 $$\textit{sign} \big( \tilde \varphi_c \circ G_c({\bf v}) \wedge  \tilde \varphi_c \circ G_c'({\bf v}') \big)=\textit{sign} \big( G_c({\bf v}) \wedge G_c'({\bf v}') \big).$$
Since, $\varphi_c$ weakly preserves orientation, it follows that $\tilde \varphi_c$ weakly preserves sign of $\wedge$. That is, if ${\bf u}$ and ${\bf w}$ are two unit vectors, then 
$$\textit{sign}\big( \tilde \varphi_c({\bf u}) \wedge \tilde \varphi_c({\bf w}) \big)=0 
\quad \textrm{ or } \quad
\textit{sign}\big( \tilde \varphi_c({\bf u}) \wedge \tilde \varphi_c({\bf w}) \big)={\bf u} \wedge {\bf w}.$$
Thus, our only fear is that $G_c({\bf v})$ and $G_c'({\bf v}')$ are distinct in $\RP^1$, but their images
$\varphi_c \circ G_c({\bf v})$ and $\varphi_c \circ G_c'({\bf v}')$ were not distinct. But directions containing saddle connections have parabolics that preserve them. So there are parabolics $P_c, P_c' \in \G^\pm_c$ which fix
$G_c({\bf v})$ and $G_c'({\bf v}')$ in $\RP^1$, respectively. The corresponding parabolics $P_1, P_1' \in \G^\pm_1$
fix the directions $\varphi_c \circ G_c({\bf v})$ and $\varphi_c \circ G_c'({\bf v}')$ in $\RP^1$. 
Now, by the commutative diagram \ref{varphi_c}, the points $\varphi_c \circ G_c({\bf v})$ and $\varphi_c \circ G_c'({\bf v}')$ in $\RP^1$
have unique preimages under $\varphi_c^{-1}$. So, 
$\varphi_c \circ G_c({\bf v})$ and $\varphi_c \circ G_c'({\bf v}')$ must be distinct after all.
\end{proof}

The next theorem does not contribute to the classification of affine automorphisms, but culminates the
discussion in this section. Suppose we have a triangulation ${\mathcal T}$ of a translation surface $S$.
The triangulation can be used to code trajectories under the geodesic flow on $S$.
We code every trajectory by the sequence of saddle connections in ${\mathcal T}$
that the trajectory crosses.
We call this the {\em orbit-type} of the trajectory.

For some ${c_1} \geq 1$ fix a  triangulation ${\mathcal T}_{c_1}$ of $S_{c_1}$. 
By theorem \ref{thm:classification_of_saddle_connections},
there are identical triangulations ${\mathcal T}_c$ of every $S_c$ with $c \geq 1$.
The saddle connections in each ${\mathcal T}_c$ are homotopic 
to the corresponding saddle connection in ${\mathcal T}_{c_1}$
up to the canonical homeomorphism $S_c \to S_{c_1}$.

\begin{theorem}
Let ${\mathcal T}_c$ be a triangulation of $S_c$ for each $c \geq 1$, as described above.
Then, given any $c_1,c_2\geq 1$ and any finite or infinite trajectory $\gamma_{1}$ 
in $S_{c_1}$, there is a trajectory $\gamma_{2}$ in $S_{c_2}$ with the same orbit type.
\end{theorem}

Again, this theorem is implied by a more general statement.

\begin{lemma}
Suppose $h:S \to S'$ is a homeomorphism between translation surfaces satisfying the three statements of
lemma \ref{lem:same_saddles}. Let ${\mathcal T}$ be a triangulation of $S$ and
${\mathcal T}'$ be the corresponding triangulation of $S'$.
Then, given any finite or infinite trajectory $\gamma$  
in $S$, there is a trajectory $\gamma'$ in $S'$ with the same orbit type.
\end{lemma}

\begin{proof}
We will prove the statement for finite trajectories. 
The statement for infinite trajectories follows
by looking at longer and longer finite subtrajectories.

Let $\gamma$ be a finite trajectory in $S$. To derive a contradiction, it would have to be that
$\gamma$ crosses at least two saddle connections in ${\mathcal T}$. Develop $\gamma$ into the plane,
along with the saddle connections it crosses. The sequence of saddle connections in the orbit-type
develop to a sequence of segments $s_i$ in the plane. We define $\Lambda$ to be the space
of all lines in $\R^2$ that cross through the interiors of each segment $s_i$. The {\em slalom hull} of the sequence 
$\langle s_i \rangle$ is 
$$\textit{SH}= \bigcup_{\ell \in \Lambda} \ell.$$
Clearly this slalom hull is non-empty, because the line containing the developed image of $\gamma$
lies in $\Lambda$. The boundary of the $\textit{SH}$ consists of four rays and a finite set of segments that
pull back to connections.
Orient $\textit{SH}$ so that it is nearly horizontal as in the figure \ref{fig:slalomhull}. 
Note that $\R^2 \smallsetminus \textit{SH}$ consists of two convex components.
We call the finite segments in the top component of $\del \textit{SH}$ the {\em top chain}.
Similarly, the segments in the bottom component will be called the {bottom chain}.
Further, note there are {\em diagonals}, which are parallel to the
infinite rays in $\del \textit{SH}$. 
They are formed by connecting the left-most
vertex of the top chain to the right-most vertex of the bottom chain, and vice versa.
We call the union of the top chain, the bottom chain, and the diagonals the {\em saddle chain},
and note that they form a loop.

\begin{figure}[h]
\begin{center}
\includegraphics[width=4.5in]{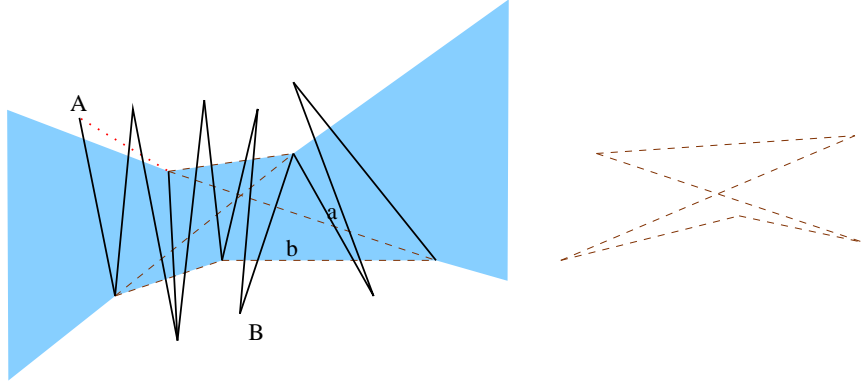}
\caption{On the left, The slalom hull of a sequence of solid segments is rendered as the shaded polygon. The slalom
chain is the sequence of dotted lines. The right shows a deformation of the slalom chain. No 
deformation of the slalom chain which preserves the cyclic ordering of edges in $\RP^2$ can destroy the slalom hull.}
\label{fig:slalomhull}
\end{center}
\end{figure}

Each $c_i$ in the saddle chain of $S$ pulls back to a saddle connection $\sigma_i$ of $S$. Let
$\sigma'_i$ be the corresponding saddle connections in $S'$. By item 3 of lemma \ref{lem:same_saddles},
the directions of saddle connections $\sigma'_i$ must be in the same cyclic order as the directions
for $\sigma_i$. We may develop the saddle connections
$\sigma'_i$ into the plane to obtain the chain of segments $c_i'$. This chain of segments
has essentially the same combinatorics. Consider the sequence of segments given by the diagonal
running from top left to bottom right, and then moving across the top chain from left to right,
and ending with the other diagonal.
Each segment must be rotated slightly counterclockwise to reach the subsequent one. Thus, the region bounded
by the top chain, by the ray leaving the top left vertex in the direction of the diagonal, and the ray leaving
the top right vertex in the direction of the other diagonal must bound a convex set.
Similarly, there is another natural convex set bounded by the lower chain and some rays. The convex sets
may not intersect, because they are guaranteed to lie in two opposite quadrants of the division of the plane by lines
through the diagonals. Thus, we have an analogous set of lines $\Lambda'$ which pass through the diagonals
and not the lower and upper chains. The slalom hull $\textit{SH}'$ given by the same formula is non-empty.

We claim that $\textit{SH}'$ is the slalom hull for the segments $\{s_i'\}$. 
This is equivalent to saying that no endpoints of $s_i'$ lie within $\textit{SH}'$, which will imply our theorem.
Some of the locations of endpoints are determined, because they are endpoints of segments
in the chain $\{c_i'\}$. 

Suppose $B \in \R^2 \smallsetminus \textit{SH}$ is a endpoint of a segment $s_i$ 
in the development of $S$ which crosses one of the edges
in the slalom hull, $b \in \del \textit{SH}$. See figure \ref{fig:slalomhull}. Let $B'$ be the corresponding 
endpoint of the segment $s_i'$ in the development from $S'$. Then $B'$ cannot be in  $\textit{SH}'$,
lest it destroy the segment $c_i'$, which must pull back to a saddle connection in $S'$. 

The more
difficult case is when $A \in \R^2 \smallsetminus \textit{SH}$ is an endpoint of a segment $s_i$ in the
development $S$ which crosses one of the rays in $\del \textit{SH}$. 
Without loss of generality, assume the ray crossed is the upper left ray in the development.
Let $a$ be the diagonal element of the slalom chain that is parallel to this ray, and $a'$ the corresponding
element in the development of $S'$. 
We consider the canonical homotopy class of $S$, with the cone singularities removed, of paths joining
the pull back of the top left vertex in $\del \textit{SH}$ to the pull back $A$. 
This homotopy class should contain the path which develops to follow the
ray of $\del \textit{SH}$ leaving the top left vertex until it hits the segment 
$s_i$ and then follows $s_i$ to $A$.
Let $d$ be the infimum of the lengths of paths in this homotopy class. There is a limiting path $p \subset S$ 
that may pass
through singularities, that achieves this infimum. This path consists of a sequence of saddle connections,
and turns only rightward in total angle less than $\pi$. The initial segment of the developed image of $p$
must immediately leave $\textit{SH}$, hence its orientation when compared to $a$ is determined.
Let $p'$ be the corresponding chain of saddle connections in $S'$, which must likewise turn rightward
by total angle less than $\pi$. The initial segment of the developed image of $p'$ must immediately
leave $\textit{SH}'$, because its orientation when compared to $a'$ must match that of the case in $S$.
Moreover, since $p'$ only bends rightward, the path cannot return to $\textit{SH}'$.
This concludes the argument that $\textit{SH}'$ is the slalom hull for the segments $\{s_i'\}$.
We know this contains lines, any of which contain a segment which pulls back to a trajectory
$\gamma'$ in $S'$ with the same orbit type as $\gamma$ in $S$. 
\end{proof}

\subsection{No other affine automorphisms}

The last step to the proof of theorems \ref{thm:veech_groups} and \ref{thm:affine_automorphisms}
is to demonstrate that all affine automorphisms of the surface lie in the group generated by the elements we listed.

\begin{lemma}
All affine automorphisms of the surface are contained in the group 
generated by $\widehat{-I}_c$, $\hat{A}_c$, $\hat{D}_c$, and $\hat{E}_c$.
\end{lemma}
\begin{proof}
Let us suppose that for some $c \geq 1$ there is an $M \in \GL(2, \R)$ in the Veech group $\Gamma(S_c)$ and
a corresponding element $\hat{M}$ in the affine automorphism group  $\Aut(S_c)$. We will prove that
$\hat{M}$ lies in the group generated by the four elements $\widehat{-I}_c$, $\hat{A}_c$, $\hat{D}_c$, and 
$\hat{E}_c$.

Let $\theta=(1,0) \in S^1$ be the horizontal direction.
We know that the image $M(\theta)$ must contain saddle connections of $S_c$. Further more
the horizontal and slope one directions can be distinguished, since the smallest area maximal cylinder in the horizontal
direction has two cone singularities in its boundary, while the smallest area maximal cylinder in the slope one direction
has four cone singularities in its boundary. Thus, by theorem \ref{thm:classification_of_saddle_connections}, there must be an
element $N_c \in \G^\pm_c$ satisfying $M(\theta) =N_c(\theta)$. It follows that $N_c^{-1} \circ M$ preserves the 
horizontal direction. 

There must be a corresponding element element $\hat{N}_c^{-1} \circ \hat{M} \in \Aut(S_c)$ with derivative
$N_c^{-1} \circ M$. The automorphism must fix the decomposition into horizontal cylinders, and fix each cylinder
in the decomposition (because the cylinders have distinct areas). The smallest area horizontal cylinder is
isometric in each $S_c$. It is built from two triangles, the convex hull of $(0,0)$, $(1,1)$, and $(-1,1)$ and
the same triangle rotated by $\pi$, with diagonal sides of the first glued to the diagonal sides of the second 
by translation. $\hat{N}_c^{-1} \circ \hat{M} \in \Aut(S_c)$ must preserve this cylinder and permute the pair of cone singularities in the boundary. Therefore
$$N_c^{-1} \circ M=\left[\begin{array}{cc} 1 & 2 n \\ 0 & 1 \end{array}\right]=D_c^n
\quad \textrm{or} \quad
N_c^{-1} \circ M=\left[\begin{array}{cc} 1 & -2 n \\ 0 & -1 \end{array}\right]=-I \circ A_c \circ D_c^n$$
for some $n \in \Z$. Therefore, $M=N_c \circ D_c^n$ or $M=N_c \circ -I \circ A_c \circ D_c^n$, all of which
lie in $\G^\pm_c$. Therefore, $\Gamma(S_c) \subset \G^\pm_c$.

Now, we have determined all of the affine automorphism group unless there is something in the kernel of the derivative
map ${\bf D}:\Aut(S_c) \to \Gamma(S_c)$. Suppose there was a non-trivial element $\hat{M}$ in the kernel. 
Then it would have to fix the smallest horizontal cylinder. 
Because there is only one cone point in the top boundary component,
$\hat{M}$ must fix every point in the smallest cylinder. Then, it must fix every point in $S_c$, since $S_c$ is path
connected.
\end{proof}

\section{Recurrence of the geodesic flow on $S_1$}
\label{sect:recurrence}
In this section, we will demonstrate the following.
\begin{theorem}[Recurrence]
\label{thm:recurrence}
The geodesic flow in irrational directions on the surface $S_1$ is recurrent.
\end{theorem}

A nice geometric consequence is that there are no strips (subsets isometric to $\R_{>0} \times (0,a)$ for $a>0$) in irrational directions.
Earlier work also tells us that there are no closed trajectories and no saddle connections in irrational directions.

The following corollary summarizes the situation in all directions.
The notation $\frac{p}{q} \equiv \frac{p'}{q'} \pmod{2}$ means that when both fractions are reduced,
$p \equiv p' \pmod{2}$ and $q \equiv q' \pmod{2}$.

\begin{corollary}[A trichotomy]
\label{cor:trichotomy}
Let $m \in \R \cup \{\infty\}$ be a slope. Then the geodesic flow in direction $m$ on $S_1$ satisfies
\begin{itemize}
\item If $m=\frac{p}{q}$ is rational and $\frac{p}{q} \equiv \frac{0}{1} \pmod{2}$ or $\frac{p}{q} \equiv \frac{1}{1}\pmod{2}$ then $m$ is a completely periodic direction. That is, there is a decomposition into cylinders in that direction and every trajectory is closed or a saddle connection. 
\item If $m=\frac{p}{q}$ is rational and $\frac{p}{q} \equiv \frac{1}{0} \pmod{2}$, then there is a decomposition
of the surface into infinite strips isometric to $I \times \R$ with $I \subset \R$ an open interval. 
Every trajectory is singular in at most one direction (forward or backward).
Any trajectory which is not singular in forward (resp. backward) time eventually leaves every compact set and never returns.
\item If $m$ is irrational, then the geodesic flow in this direction is recurrent. However, there are no closed trajectories or saddle connections
in this direction.
\end{itemize}
\end{corollary}
The case for rational directions follows from earlier work. Every rational direction is the image under $\Gamma(S_1)$,
the congruence two subgroup of $\SL(2,\Z)$, of the horizontal, vertical, or slope one directions. 

Now we will discuss the proof of theorem \ref{thm:recurrence}. Given an irrational direction $\theta$ on the surface, 
we get a foliation of the surface by leaves which travel in the direction $\theta$. There is a natural transverse
measure $\mu_\theta$ on the leaves, which measures the Euclidean width of leaves a transversal crosses. 
For each irrational $\theta$ we will construct a nested sequence of subsurfaces 
$D_1 \subset D_2 \subset D_3 \subset \ldots$ satisfying the following lemma.

\begin{lemma}[Nested exhaustion]
\label{lem:nested}
Suppose $S$ is a translation surface and $\theta$ is a direction.
Suppose also there exists a nested sequence of finite area compact
subsurfaces $D_1 \subset D_2 \subset D_3 \subset \ldots$ 
satisfying the following two conditions.
\begin{enumerate}
\item $S=\bigcup_{n \in \Z} D_n$.
\item $\liminf_{n \to \infty} \mu_\theta( \del D_n) = 0$.
\end{enumerate}
Then the geodesic flow in direction $\theta$ is recurrent. 
\end{lemma}

This lemma is not a new idea. It is used in work of Gutkin and Troubetzkoy, where it is used to prove that
the geodesic flow in right triangles is directionally recurrent \cite{GT95}.

\begin{proof}[Proof of the nested exhaustion lemma]
Suppose the conditions of the lemma are satisfied. Fix a finite transversal $I$ to the foliation in the direction
$\theta$, we will show that there is  a first return map under the flow in direction $\theta$ which is well defined almost everywhere. 
We can find an $N$ so that $I \subset D_N$. Given any $n \geq N$ we define a transverse measure
preserving map on a full measure subset $J \subset I$ of the form $\phi_n:J \to I \cup \del D_N$. Given a point 
$x \in I$, simply flow in direction $\theta$ until you hit either $I$ or $\del D_n$, the image $\phi_n(x)$ is that
point that you first hit. This is not defined for every point, because the trajectory of a point might never hit 
$I$ or $\del D_n$. But, because $D_n$ is finite area, it is well defined on a subset of full measure. Using the fact
that $\phi_n$ is measure preserving we see
$$\mu_\theta \big( \phi_n^{-1}(I) \big)>\mu_\theta(I) - \mu_\theta(\del D_n).$$
Note, by condition $2$ of the lemma, there is a subsequence of times $n$ so that the right side converging to $\mu_\theta(I)$ as $n \to \infty$. Furthermore,
if $m>n$ and $x$ is a point so that $\phi_n(x) \in I$ then $\phi_m(x)=\phi_n(x)$. Thus, it makes sense
to define a limiting function 
$\phi_\infty:I \to I$ where $\phi_\infty(x)=y$ if there is an $n$ with $\phi_n(x)=y$. Because of the above
argument, this function is defined on a subset of full measure and realizes the first return map to $I$.
Thus the flow in direction $\theta$ is recurrent.
\end{proof}

In the next two subsections we will prove that this lemma applies to our situation.

\subsection{Subsurfaces built from cylinders}

The maximal cylinders in a completely periodic direction of $S_1$ 
have a canonical numbering by $\N$ in order of increasing areas. We identify a rational direction using a reduced
fractional slope $\frac{p}{q} \in \RP^1$, representing the direction the vector $(q,p) \in \R^2$ points.
We allow this to include the fraction $\infty=\frac{1}{0}$.
We will use $\Cyl_{q,p}^n$ to denote the $n$-th cylinder in the direction $\frac{q}{p}$.

Let ${\mathcal D}_{q,p}^n$ denote the closure of the union of the cylinders $\Cyl_{q,p}^1, \Cyl_{q,p}^2, \ldots \Cyl_{q,p}^n$. Note that by definition we have ${\mathcal D}_{q,p}^n \subset {\mathcal D}_{q,p}^{n+1}$.
The boundaries of the surfaces ${\mathcal D}_{q,p}^n$ are described by the following proposition.

\begin{proposition}
\label{prop:boundaries}
Let $\frac{p}{q} \not \equiv \frac{1}{0} \pmod{2}$. The boundary of the surface ${\mathcal D}_{q,p}^n$
consists of two saddle connections. When $\frac{p}{q} \equiv \frac{0}{1} \pmod{2}$,
they each have holonomy $(2 n q,2 n p)$. When $\frac{p}{q} \equiv \frac{1}{1} \pmod{2}$,
they each have holonomy $\big((2 n+1) q,(2 n+1) p\big)$.
\end{proposition}
\begin{proof}
The statement is invariant under the affine automorphism group. Thus the statement only needs to be checked in
the horizontal and slope one directions. We leave it to the reader to check these two cases.
\end{proof}

The {\em Farey graph} ${\bf F}$ is a graph whose vertices are the rational points of $\RP^1$ (i.e. the vertices
are points in $\Q \P^1)$. 
Two rational points
$\frac{p}{q}$ and $\frac{r}{s}$ are the vertices of an edge of the Farey graph if $ps-rq= \pm 1$. Geometrically,
${\bf F}$ may be realized as an $\SL^{\pm}(2,\Z)$ invariant triangulation of the hyperbolic plane by ideal triangles.

The following proposition summarizes the containment properties of the subsurfaces of the form ${\mathcal D}_{q,p}^n$.

\begin{proposition}
\label{prop:containment}
Assume the fractions $\frac{p}{q} \equiv \frac{0}{1} \pmod{2}$ and $\frac{r}{s} \equiv \frac{1}{1} \pmod{2}$
are adjacent in the Farey graph. Then for all $n \in N$,
$${\mathcal D}_{q,p}^n \subset {\mathcal D}_{s,r}^n
\quad \textrm{and} \quad
{\mathcal D}_{s,r}^n \subset {\mathcal D}_{q,p}^{n+1}$$
\end{proposition}
\begin{proof}
By theorem \ref{thm:veech_groups}, the Veech group of $S_1$ is the congruence two subgroup
of $\SL(2,\Z)$. This group acts transitively on edges of the form mentioned in the proposition. The
statement is affine automorphism invariant, so it is sufficient to prove it for the case of $\frac{p}{q}=\frac{0}{1}$
and $\frac{r}{s}=\frac{1}{1}$. This case is illustrated by figure \ref{fig:surface_cylinders}.
\end{proof}

This proposition encapsulates all we know about the containment properties of the subsurfaces 
${\mathcal D}_{q,p}^n$. So, the Farey graph contains many edges that will be of no use to us. We define
the subgraph ${\bf G} \subset {\bf F}$ to be the graph obtained by removing vertices corresponding
to rationals equivalent to $\frac{1}{0}$ modulo $2$. All edges which connect to a vertex equivalent to $\frac{1}{0}$ modulo $2$ must also be removed. The graph ${\bf G}$ is depicted in figure \ref{fig:graph}.

\begin{figure}[h]
\begin{center}
\includegraphics[width=4.5in]{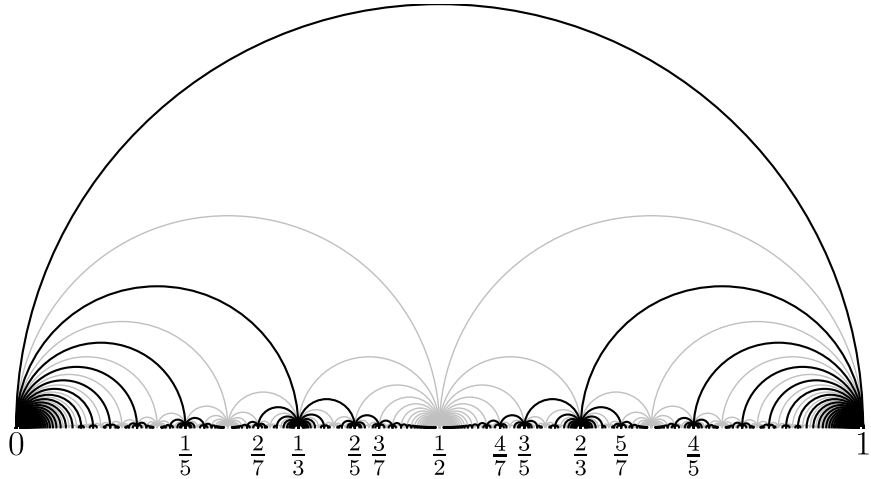}
\caption{The dark lines illustrate the portion of ${\bf G}$ within $[0,1]$. The Farey graph
includes the lighter lines. Both graphs are invariant under integer translations. }
\label{fig:graph}
\end{center}
\end{figure}

The graph ${\bf G}$ is a tree. In fact, it is the tree with countablly many edges meeting at every vertex. 
Thus, paths between two vertices which do not
backtrack over themselves are unique. Furthermore, if $\frac{p_1}{q_1} \in {\bf G}$ is a vertex and 
$\theta \in \R \subset \RP^1$ 
is irrational, there is a unique non-backtracking path so that the vertices crossed by the path 
$\frac{p_1}{q_1}, \frac{p_2}{q_2}, \ldots$ converge to $\theta$. When $\frac{p_1}{q_1}=n$ is the closest integer
to $\theta$, we call this sequence the ${\bf G}$-sequence. 

There is a natural coding of the ${\bf G}$-sequence converging to $\theta$ of the form 
$(\frac{p_1}{q_1}, \frac{p_2}{q_2} ; k_2, k_3, \ldots)$.  $\frac{p_1}{q_1}$ and $\frac{p_2}{q_2}$
are the first two rationals in the ${\bf G}$-sequence. Each $k_i \in \Z \smallsetminus {0}$. 
For $i \geq 2$, assume the $i-1$-st and $i$-th fractions in the ${\bf G}$-sequence are given by
$\frac{p_{i-1}}{q_{i-1}}$ and $\frac{p_i}{q_i}$. Then all the rationals which are adjacent in the Farey graph to 
$\frac{p_i}{q_i}$ are of the
form $\frac{p_{i-1}+j p_i}{q_{i-1}+j q_i}$. Thus, there is a $j$ so that the $i+1$-st rational is given 
by $\frac{p_{i+1}}{q_{i+1}}=\frac{p_{i-1}+j p_i}{q_{i-1}+j q_i}$. In order for this fraction not to be congruent 
to $\frac{1}{0}$ modulo two, the number $j$ must be even. To avoid backtracking, $j \neq 0$. 
So, we set add the number $k_i=\frac{j}{2}$ to our coding. This number satisfies
\begin{equation}
\label{eq:inductive_G_sequence}
\frac{p_{i+1}}{q_{i+1}}=\frac{p_{i-1}+2 k_i p_i}{q_{i-1}+2 k_i q_i},
\end{equation}
where we simplify fractions so that the denominator is always positive.

Now we will describe the nested sequence of surfaces used to apply lemma \ref{lem:nested}. Let 
$\theta \in \R \subset \RP^1$ be any irrational slope and consider the ${\bf G}$-sequence of
$\theta$, $\langle \frac{p_i}{q_i} \rangle$. We consider the sequence of 
surfaces $\langle D_{q_i,p_i}^i \rangle_{i\in \N}$. By proposition \ref{prop:containment} each surface
is contained in the subsequent surface. Using proposition \ref{prop:boundaries}, we compute that
$$\mu_\theta( \del D_{q_i,p_i}^i)=\begin{cases}
\frac{2 i}{\sqrt{1+\theta^2}} \big|(1,\theta) \wedge (q_i,p_i)\big| 
& \textrm{if $\frac{p_i}{q_i} \equiv \frac{0}{1} \pmod{2}$} \\
\frac{2 i+1}{\sqrt{1+\theta^2}} \big|(1,\theta) \wedge (q_i,p_i)\big|
& \textrm{if $\frac{p_i}{q_i} \equiv \frac{1}{1} \pmod{2}$}.
\end{cases}$$
We will find it convenient to notice that
\begin{equation}
\label{eq:first_inequality}
\mu_\theta( \del D_{q_i,p_i}^i) \leq 
\frac{2 i+1}{\sqrt{1+\theta^2}} \big|(1,\theta) \wedge (q_i,p_i)\big|=
\frac{(2 i+1) q_i}{\sqrt{1+\theta^2}} | \theta - \frac{p_i}{q_i}|
\end{equation}
In order to apply lemma \ref{lem:nested}, it is sufficient to show that the $\lim \inf$ of the right hand side as $i \to \infty$ is zero. We see this sufficient condition is fulfilled if 
\begin{equation} 
\label{eq:inequality_to_prove}
\liminf_{i \to \infty} i q_i | \theta - \frac{p_i}{q_i}| =0,
\end{equation}
since there is a constant $\kappa$ depending only on $\theta$ 
so that $\kappa i q_i>\frac{(2 i+1) q_i}{\sqrt{1+\theta^2}}$.
Taking $\kappa=\frac{3}{\sqrt{1+\theta^2}}$ suffices.
We will prove equation \ref{eq:inequality_to_prove} is true in the next subsection. The following lemma is
the main contribution of our knowledge of the ${\bf G}$-sequence.

\begin{lemma}
\label{lem:g_sequence}
Let $\langle \frac{p_i}{q_i} \rangle$ be the ${\bf G}$-sequence for an irrational $\theta$. Then
$$\lim_{i \to \infty} \frac{i}{q_i}=0.$$
\end{lemma}
\begin{proof}
Consider the recursive definition for the denominator given in equation \ref{eq:inductive_G_sequence}.
$$q_{i+1}=|q_{i-1}+2 k_i q_i|.$$
It can be seen inductively that the sequence of denominators is strictly increasing. Furthermore,
\begin{equation}
\label{eq:difference}
q_{i+1} - q_{i} \geq q_{i}-q_{i-1}.
\end{equation}
Equality only happens when $k_i=-1$. We break into two cases. 

First, suppose that infinitely many $k_i$ are not equal to $-1$. Fix an arbitrary $c>0$. 
Then, by the previous paragraph, there is a $N>0$ such that $q_{N+1}-q_N>c$. Then by equation
\ref{eq:difference} we know $q_{N+j} > q_N+jc$ for all $j \geq 1$. Consequently, 
$$\limsup_{i \to \infty} \frac{i}{q_i}<\limsup_{j \to \infty} \frac{N+j}{q_N+jc}=\frac{1}{c}.$$
Since $c$ was arbitrary the conclusion of the lemma must hold. 

If the first case does not hold, there is an $N\in N$ such that $k_i=-1$ for all $i \geq N$. So for 
$i \geq N$, we obtain the inductive formula
$$q_{i+1}=2 q_i - q_{i-1} \quad \textrm{and} \quad p_{i+1}=2 p_i - p_{i-1}.$$
It follows that for all $j\geq 0$, we know $q_{N+j}=q_N+j(q_N-q_{N-1})$ and $p_{N+j}=p_N+j(p_N-p_{N-1})$.
Then
$$\theta=\lim_{j \to \infty} \frac{p_{N+j}}{q_{N+j}}=
\lim_{j \to \infty} \frac{p_N+j(p_N-p_{N-1})}{q_N+j(q_N-q_{N-1})}=\frac{p_N-p_{N-1}}{q_N-q_{N-1}}.$$
But this contradicts the assumption that $\theta$ is irrational. That is, the coding of the ${\bf G}$-sequence for
an irrational can not end in an infinite sequence of $-1$s. 
\end{proof}

\subsection{Continued fractions}

We will now introduce continued fractions.
By the notation $[a_0; a_1, a_2, \ldots, a_k]$ we mean the continued fraction
$$[a_0; a_1, a_2, \ldots, a_k]=a_0 + \cfrac{1}{a_1+\cfrac{1}{a_2+ \cfrac{}{\ddots a_{k-1} + \cfrac{1}{ 1+\frac{1}{a_k}}}}}.$$
See \cite{Khinchin} and \cite{Rockett-Szusz} for introductions to continued fractions.

Associated to an irrational $\theta$ is a unique sequence of rational approximates $\frac{s_i}{t_i}$, called
{\em the continued fraction approximates}, 
a sequence of integers $[a_0; a_1, a_2, \ldots]$, which is the {\em continued fraction expansion},
and a sequence of irrational {\em remainders} $r_0, r_1, \ldots$.
For purposes of the inductive argument, we define $s_{-1}=1$ and $t_{-1}=0$. 
We define $a_0$ to be the largest integer less than $\theta$, $\frac{s_0}{t_0}=\frac{a_0}{1}$, and $r_0=\theta$.
It will inductively be true that
\begin{equation}
\label{eq:inductive_hypothesis}
0<r_i-a_i<1.
\end{equation}
Given definitions of $a_i$, $\frac{s_i}{t_i}$, and $r_i$, we define $r_{i+1}$ to be the unique number satisfying
$$r_i=a_i+\frac{1}{r_{i+1}}.$$
By our inductive hypothesis, $r_{i+1}>1$. We choose $a_{i+1}$ to be the largest integer less than $r_{i+1}$, so that 
$a_{i+1}$ and $r_{i+1}$ satisfy the inductive hypothesis, equation \ref{eq:inductive_hypothesis}. Then we define
$$\frac{s_{i+1}}{t_{i+1}}=[a_0; a_1, \ldots, a_{i+1}]$$
When expanded, we get the more useful formula
$$\frac{s_{i+1}}{t_{i+1}}=\frac{a_{i+1}s_{i}+s_{i-1}}{a_{i+1}t_{i}+t_{i-1}}.$$

It is well known that the continued fraction sequence for $\theta$ converges to $\theta$. Moreover, we have the estimate
\begin{equation}
\label{eq:cf_estimate}
|\theta-\frac{s_i}{t_i}|<\frac{1}{q_i^2}.
\end{equation}

\begin{proposition}
Infinitely many fractions in the continued fraction sequence for an irrational $\theta$ also lie in the ${\bf G}$-sequence of $\theta$.
\end{proposition}
\begin{proof}
A pair of adjacent continued fraction approximates can be verified to lie on opposite sides of the number $\theta$. Also, they are adjacent
in the Farey graph. Choose $i$ large enough so that $t_i>1$. Then, the interval with endpoints
$\frac{s_i}{t_i}$ and $\frac{s_{i+1}}{t_{i+1}}$ contains $\theta$, but no integers. Let $E$ denote the edge of the Farey graph with endpoints $\frac{s_i}{t_i}$ and $\frac{s_{i+1}}{t_{i+1}}$. Any path through the Farey graph starting at an integer, which limits on $\theta$ must pass
through $E$. Therefore, one of the endpoints of $E$ must appear in the ${\bf G}$-sequence. We conclude that at least one of every adjacent pair of continued fraction approximates must be in the ${\bf G}$-sequence.
\end{proof}

\begin{proof}[Proof of theorem \ref{thm:recurrence}]
Following the logic of the previous subsection, to apply lemma \ref{lem:nested}, it is sufficient to prove equation \ref{eq:inequality_to_prove}.
Let $\frac{p_{i(j)}}{p_{i(j)}}$ denote the subsequence of the ${\bf G}$-sequence, $\frac{p_i}{q_i}$, for $\theta$ which are in the continued 
fraction sequence for $\theta$. This is a subsequence because of the previous proposition. Using the estimate \ref{eq:cf_estimate}
we may simplify the left hand side of equation \ref{eq:inequality_to_prove}.
$$\liminf_{i \to \infty} i q_i | \theta - \frac{p_i}{q_i}| \leq \liminf_{j \to \infty} i(j) q_{i(j)} | \theta - \frac{p_{i(j)}}{q_{i(j)}}|
\leq \liminf_{j \to \infty} \frac{i(j)}{q_{i(j)}}=0$$
The last step follows directly from lemma \ref{lem:g_sequence}.
\end{proof}

\section{Dynamics of the affine automorphism group}
\label{sect:G_dynamics}
This section concerns asymptotics of areas of cylinder intersections under the action of the affine automorphism
group. This turns out to be roughly equivalent to studying the action of the affine automorphism on homology.

This section culminates with the proof of equation \ref{eq:asymptotic_formula} of the introduction,
which is restated as theorem \ref{thm:asymptotics_of_cylinder_intersections}.

\subsection{Intersections of cylinders}

Consider two cylinders ${\mathcal A}$ and ${\mathcal B}$ in the surface $S_1$,
and an element $\hat{H}$ in the affine automorphism group $\Aut(S_1)$ acting hyperbolically. 
Algebraically, acting hyperbolically means the derivative $H={\bf D}(\hat{H})$ in the Veech group
has an eigenvalue bigger than one. 
We are interested in studying the asymptotics of the quantity
$$\textit{Area}( \hat{H}^n({\mathcal A}) \cap {\mathcal B})$$
as $n$ tends to infinity. 

Let $\gamma_{\mathcal A}$ and $\gamma_{\mathcal B}$ be core curves of the cylinders ${\mathcal A}$ and ${\mathcal B}$
respectively.  Given a closed curve on a translation surface, we can develop it into the plane. This development 
is canonical up to translations. The translation vector taking the initial point of this development to the 
termination point is the {\em holonomy} around the closed curve. Let
$\hol(\gamma_{\mathcal A}), \hol(\gamma_{\mathcal B}) \in \R^2$ denote the holonomy around the core curves.
The wedge product between two vectors in $\R^2$ is the usual one described by equation \ref{eq:wedge}.
This is also the signed area of the parallelogram formed by the two vectors.
\begin{proposition}
\label{prop:intersection_formula}
Two non-parallel cylinders ${\mathcal A}$ and ${\mathcal B}$ on a translations surface satisfy the equation
$$\textit{Area}({\mathcal A} \cap {\mathcal B})=
\frac{\big|\gamma_{\mathcal A} \cap \gamma_{\mathcal B})\big|\textit{Area}({\mathcal A}) \textit{Area}({\mathcal B})}
{\big|\hol(\gamma_{\mathcal A}) \wedge \hol(\gamma_{\mathcal B})\big|},
$$
where $|\gamma_{\mathcal A} \cap \gamma_{\mathcal B})|$ denotes the absolute value of the algebraic intersection
number.
\end{proposition}

\begin{figure}[htbp]
\begin{center}
\includegraphics[width=3in]{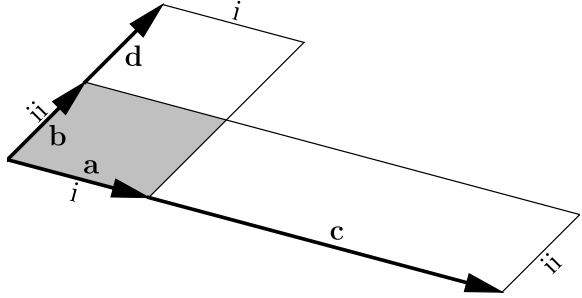}
\caption{Two intersecting cylinders developed into the plane. Roman numerals indicate edge identifications, which reconstruct the cylinders.}
\label{fig:cylinder_intersection}
\end{center}
\end{figure}

\begin{proof}
Cylinders on a translation surface intersect in a union of parallelograms, which are isometric and differ only by parallel translation. The number of these parallelograms is the absolute value of the algebraic intersection number between the core curves. Thus, we need to show that the area of one such parallelogram is given by 
\begin{equation}
\label{eq:parallelagram}
\frac{\textit{Area}({\mathcal A}) \textit{Area}({\mathcal B})}
{\big|\hol(\gamma_{\mathcal A}) \wedge \hol(\gamma_{\mathcal B})\big|}.
\end{equation}
Develop the two cylinders into the plane from an intersection as in figure \ref{fig:cylinder_intersection}. Define the vectors ${\bf a}$, ${\bf b}$, ${\bf c}$, and ${\bf d}$ as in the figure. Then the area of the two cylinders is given by the quantities
$$\textit{Area}({\mathcal A})=|\v{a} \wedge (\v{b}+\v{d})| \quad \text{ and} \quad \textit{Area}({\mathcal B})=|(\v{a}+\v{c}) \wedge \v{b}|.$$
Since the pair of vectors $\v{a}$ and $\v{c}$ are parallel as are the pair $\v{b}$ and $\v{d}$, we may write the product of areas as
$$\textit{Area}({\mathcal A})\textit{Area}({\mathcal B})=
|\v{a} \wedge \v{b}| |(\v{a}+\v{c}) \wedge (\v{b}+\v{d})|.$$
The wedge of the holonomies may be written as
$$\big|\hol(\gamma_{\mathcal A}) \wedge \hol(\gamma_{\mathcal B})\big|=|(\v{a}+\v{c}) \wedge (\v{b}+\v{d})|.$$
Thus, the quotient given in equation \ref{eq:parallelagram} is $|\v{a} \wedge \v{b}|$, the area of the 
parallelogram formed by the intersection.
\end{proof}

Let us apply this proposition to our problem. We use the facts that 
$\hat{H} \big( \hol(\gamma_{\mathcal A})\big)=\hol\big(\hat{H}(\gamma_{\mathcal A})\big)$, and that $\hat{H}$ preserves area.
Thus we have,
\begin{equation}
\textit{Area}( \hat{H}^n({\mathcal A}) \cap {\mathcal B})=
\frac{\big|\hat{H}_\ast^n(\hom{\gamma_{\mathcal A}}) \cap \hom{\gamma_{\mathcal B})} \big|}
{\big|H^n(\hol(\gamma_{\mathcal A})) \wedge \hol(\gamma_{\mathcal B})\big|}
\textit{Area}({\mathcal A}) \textit{Area}({\mathcal B}).
\end{equation}
Here $\hat{H}_\ast$ denotes the action of $\hat{H}$ on the first homology group, 
$\hom{\gamma_{\mathcal A}}$ and 
$\hom{\gamma_{\mathcal B}}$ are the homology classes of the curves,
and $H \in \SL(2,\R)$ is the derivative of $\hat{H}$.
The denominator of this expression is easily understood, but the numerator will require some real work.

\subsection{Homological Spaces}

Our surfaces have infinite genus. Therefore, we must be careful by what we mean by homology.

We use $S_{\geq 1}$ to denote any of the homeomorphic surfaces $S_c$, when we only care about topology.
We will be considering the homological space $H_1(S_{\geq 1} \smallsetminus \Sigma, \Z)$, 
which we define to be the linear space generated by finite weighted sums
homology classes of closed curves in $S_{\geq 1}$ with the singularities removed. 

Let $\{\Cyl_{1,0}^{i} ~|~i \in \N \}$ be the collection of horizontal cylinders ordered by increasing size and oriented rightward. 
Similarly, $\{\Cyl_{1,1}^{i} ~|~i \in \N \}$ will be the collection of cylinders whose core curves have slope $1$
and oriented in the direction $(1,1)$. We will use $\hom{\Cyl_{p,q}^i}$ to denote the homology class corresponding to the core curve of cylinder $\Cyl_{p,q}^i$ oriented in direction $(p,q) \in \R^2$.

\begin{proposition}[Generators for homology]
Elements of $H_1(S_{\geq 1} \smallsetminus \Sigma, \Z)$ can be written uniquely as a finite linear combination of elements of 
$\{\hom{\Cyl_{1,0}^{i}}\}_{i \in \N }$ and $\{\hom{\Cyl_{1,1}^{i}}\}_{i \in \N }$.
\end{proposition}

That is, cylinder which are of slope zero or one form a countable basis for $H_1(S_{\geq 1} \smallsetminus \Sigma, \Z)$.
Figure \ref{fig:surface_cylinders} illustrates the decompositions into horizontal and slope one cylinders.

\begin{proof}
As in figure \ref{fig:s1}, we can cut $S_{\geq 1}$ into two polygonal pieces along countably many saddle connections.
The resulting pieces are simply connected, so a homology class is determined by the signed number of times it crosses each of these saddle connections. Number the saddle connections $\{\sigma_{i}\}_{i\in\Z}$
counterclockwise from the point of view of the top piece so that the saddle connection 
of slope one in figure \ref{fig:s1} is numbered $\sigma_{0}$. It is sufficient to show that there is a unique way to 
write the curve $\gamma_i$ for $i \neq 0$ which enters the top piece through $\sigma_{0}$ then exits
the top piece through $\sigma_{i}$ then travels through the bottom piece to $\sigma_{0}$ and closes up. The
only way to write $\hom{\gamma_i}$ is by 
$$\hom{\gamma_i}=\begin{cases}
\sum_{j=1}^{-i-1} \hom{\Cyl_{1,1}^j} - \sum_{j=1}^{-i} \hom{\Cyl_{1,0}^j} & \textrm{if $i<0$} \\
\sum_{j=1}^{i} \hom{\Cyl_{1,1}^j} - \sum_{j=1}^{-i} \hom{\Cyl_{1,0}^j}  & \textrm{if $i>0$.} 
\end{cases}
$$
To see that this is the unique way to write $\hom{\gamma_i}$, note that there is only one horizontal or slope
one cylinder that intersects $\sigma_{0}$, and only two of these cylinders intersect each of the remaining 
$\sigma_i$.
\end{proof}

We will use the pair $\hp{\alpha}{\beta}$ where $\alpha=\langle \alpha_n \rangle_{n \in N}$ and $\beta=\langle \beta_n \rangle_{n \in N}$ 
are integral sequences which have bounded support (eventually all terms are zero) to denote the homology class
$$
\hp{\alpha}{\beta}= \sum_{n \in \N} \alpha_n \hom{\Cyl_{1,0}^n} + \sum_{n \in \N} \beta_n \hom{\Cyl_{1,1}^n}.$$
By the theorem, every element of $H_1(S_{\geq 1} \smallsetminus \Sigma, \Z)$ can be written uniquely in this way.

Given two homology classes, we can compute the algebraic intersection number between the two. 
This gives rise to an integral symplectic bilinear form on $H_1(S_{\geq 1} \smallsetminus \Sigma, \Z)$. 
It is not hard to check that $\cap:H_1(S_{\geq 1} \smallsetminus \Sigma, \Z) \times H_1(S_{\geq 1} \smallsetminus \Sigma, \Z) \to \Z$ is given by the formula
\begin{equation}
\label{eq:intersection_form}
\hp{\alpha^x}{\beta^x} \cap \hp{\alpha^y}{ \beta^y} = 
\sum_{n \in \Z} (\alpha^x_n \beta^y_n+\alpha^x_{n+1} \beta^y_n-\beta^x_n \alpha^y_n-\beta^x_n \alpha^y_{n+1})
\end{equation}

The translation surface structures on $S_{\geq 1}$, given by $S_c$ for $c \geq 1$, yield more structure.
The holonomy depends only on the homology class of the curve and 
acts linearly. Thus, for each $c \geq 1$ we have the holonomy maps
$\hol_c: H_1(S_{\geq 1} \smallsetminus \Sigma, \Z) \to \R^2$. 
These maps are naturally cohomology classes, but we will prefer to think of them as homology classes by 
using the intersection form. 
For instance, the vertical holonomy of a curve can be computed applying the intersection pairing
to the homology class of the curve and a (countable) sum of homology classes of the horizontal
cylinders weighted by the cylinder's widths. This requires more than a finite weighted sum of generators,
so we introduce the larger space 
$\HH_1(S_{\geq 1} \smallsetminus \Sigma, \R) \supset H_1(S_{\geq 1} \smallsetminus \Sigma, \R)$.
We define $\HH_1(S_{\geq 1} \smallsetminus \Sigma, \R)$ to be the space of all pairs
of sequences $\hp{\alpha}{\beta}$ of the form
$\alpha=\langle \alpha_n \in \R \rangle_{n \in N}$ and $\beta=\langle \beta_n \in \R \rangle_{n \in N}$.
We have no condition of bounded support on elements of $\HH_1(S_{\geq 1} \smallsetminus \Sigma, \R)$.
Equation \ref{eq:intersection_form} extends $\cap$ to a bilinear map
$$\cap:\HH_1(S_{\geq 1} \smallsetminus \Sigma, \R) \times H_1(S_{\geq 1} \smallsetminus \Sigma, \R) \to \R.$$

\begin{proposition}
\label{prop:hol}
For all $c \geq 1$, define the linear map ${\mathcal L}_c:\R^2 \to \HH_1(S_{\geq 1} \smallsetminus \Sigma, \R)$
by ${\mathcal L}_c:(a,b) \mapsto \hp{\alpha}{\beta}$, where $\hp{\alpha}{\beta}$ is given
coordinatewise by polynomials in $c=\cos \theta$ as follows.
$$\alpha_n=(a-b) \textit{Width}(\Cyl_{1,0}^n)=(a-b) \sum_{j=-(n-1)}^{n-1} \cos j\theta$$
$$\beta_n=b\textit{Width}(\Cyl_{1,1}^n)\sqrt{2}=\frac{
(2b) \sum_{j=-(n-1)}^n \cos j\theta}{1+\cos \theta}$$
 ${\mathcal L}_c$ satisfies the identity ${\mathcal L}_c(a,b) \cap \hom{x} = (a,b) \wedge \hol_c(\hom{x})$ for all
 $(a,b) \in \R^2$.
\end{proposition}

\begin{remark}
In the above proposition, the quantities $\alpha_n$ and $\beta_n$ should be viewed as polynomials in
$c=\cos \theta$ (e.g. $\cos (-2 \theta)=2 \cos^2 \theta-1=2 c^2-1$). The term $1+\cos \theta$ 
in the denominator of $\beta_n$ cancels with a factor from the numerator. Since all
$\alpha_n$ and $\beta_n$ are polynomials in $c$, it makes sense to define the expression 
${\mathcal L}_c(a,b)$ for all $c \in \C$.
\end{remark}

\begin{proof}[Proof of proposition \ref{prop:hol}]
It is sufficient to prove the statement on pairings basis of $\R^2$. We choose the basis 
$\{(1,0),(\frac{\sqrt{2}}{2},\frac{\sqrt{2}}{2}) \}$.
The quantity $(1,0)\wedge \hol_c(\hom{x})$ measures the $y$-coordinate of $\hol_c(\hom{x})$.
As mentioned in the paragraph before the proposition, 
this quantity may be computed by intersecting with the homology class 
$$\sum_{n=1}^{\infty} \textit{Width}(\Cyl_{1,0}^n) \hom{\Cyl_{1,0}^n}.$$
After some trigonometric simplifications, it can be checked that this is the same as ${\mathcal L}_c(1,0)$.
Similarly, $(\frac{\sqrt{2}}{2},\frac{\sqrt{2}}{2})\wedge \hol_c(\hom{x})$
measures distance in the slope $-1$ direction. This can be computed by intersecting with
$$\sum_{n=1}^{\infty} \textit{Width}(\Cyl_{1,1}^n) \hom{\Cyl_{1,1}^n}.$$
Again, this is the same as ${\mathcal L}_c(\frac{\sqrt{2}}{2},\frac{\sqrt{2}}{2})$.
\end{proof}

We can define the first cohomological group $H^1(S_{\geq 1} \smallsetminus \Sigma, \R)$ as
the set of all linear maps $\phi:H_1(S_{\geq 1} \smallsetminus \Sigma, \R) \to \R$. 
Via $\cap$, there is a natural map from
$\Phi:\HH_1(S_{\geq 1} \smallsetminus \Sigma, \R) \to H^1(S_{\geq 1} \smallsetminus \Sigma, \R)$ namely given $\hom{x} \in \HH_1(S_{\geq 1} \smallsetminus \Sigma, \R)$ take the map
$$\Phi(\hom{x}):H_1(S_{\geq 1} \smallsetminus \Sigma, \R) \to \R: \hom{y} \mapsto \hom{x} \cap \hom{y}.$$
But $\Phi$ is not injective. It can be checked that $\Phi$ has a one dimensional kernel spanned by the element
\begin{equation}
\label{eq:kernel}
\hom{z}=\hp{\langle 1, -1, 1, -1, \ldots \rangle}{ \langle 0, 0, 0, \ldots \rangle}.
\end{equation}

\subsection{Action of automorphisms on homology}
By theorem \ref{thm:affine_automorphisms}, the actions of the affine automorphism groups $\Aut(S_c)$
on $S_{\geq 1}$ are the same up to isotopy. Therefore, it makes sense to use $\Aut(S_{\geq 1})$ to denote
this group action, when we act on homology. We will describe the action of generators of $\Aut(S_{\geq 1})$
on $H_1(S_{\geq 1} \smallsetminus \Sigma, \Z)$.

If $\{\gamma_i\}_{i \in \Lambda}$ is a countable collection of disjoint curves on a surface $S$, then the action of
a right Dehn twist on homology is given by the map
$$H_1(S \smallsetminus \Sigma, \Z) \to H_1(S \smallsetminus \Sigma, \Z): 
\hom{x} \mapsto \hom{x}+\sum_{i \in \Lambda} (\hom{\gamma_i} \cap \hom{x}) \hom{\gamma_i},$$
which is independent of the orientations assigned to the $\gamma_i$.

The action of the element $\hat{D} \in \Aut(S_{\geq 1})$ on homology is given by
$$
\begin{array}{rcl}
\hat{D}_\ast & : &  H_1(S_{\geq 1} \smallsetminus \Sigma, \Z) \to H_1(S_{\geq 1} \smallsetminus \Sigma, \Z) \\
& : & \hom{x} \mapsto \hom{x}+\sum_{i \in \N} (\hom{\Cyl^i_{1,0}} \cap \hom{x}) \hom{\Cyl_{1,0}^i}.
\end{array}$$
If we take $\hp{\alpha}{\beta} \in H_1(S_{\geq 1} \smallsetminus \Sigma, \Z)$ and set 
$\hp{\alpha'}{\beta'}=\hat{D}_\ast(\hp{\alpha}{\beta})$ then
\begin{equation}
\alpha'_n=\alpha_n+\beta_{n-1}+\beta_{n} \quad \textrm{ and } \quad \beta'_n=\beta_n,
\end{equation}
where we interpret $\beta_0=0$ (since $0 \not\in \N$).

Similarly, the action of $\hat{E} \in \Aut(S_{\geq 1})$ on homology is 
$$\begin{array}{rcl}
\hat{E}_\ast & : &  H_1(S_{\geq 1} \smallsetminus \Sigma, \Z) \to H_1(S_{\geq 1} \smallsetminus \Sigma, \Z) \\
& : & \hom{x} \mapsto \hom{x}+\sum_{i \in \N} (\hom{\Cyl^i_{1,1}} \cap \hom{x}) \hom{\Cyl_{1,1}^i}.
\end{array}$$
If we set $\hp{\alpha'}{\beta'}=\hat{E}_\ast(\hp{\alpha}{\beta})$, then 
\begin{equation}
\alpha'_n=\alpha_n \quad \textrm{ and } \quad \beta'_n=\beta_n-\alpha_n-\alpha_{n+1}.
\end{equation}

Finally, we describe the action of $\hat{A} \in \Aut(S_{\geq 1})$ on homology. If 
$\hp{\alpha'}{\beta'}=\hat{A}_\ast(\hp{\alpha}{\beta})$ then
\begin{equation}
\label{eq:A-action}
\alpha'_n=-\alpha_n -\beta_{n-1}-\beta_n \quad \textrm{ and } \quad \beta'_n=\beta_n.
\end{equation}

The action of every $\hat{G} \in \Aut(S_{\geq 1})$ on homology is {\em sparse} in the sense that the
formula for the $\alpha'_n$ and $\beta'_n$ of terms
$\hp{\alpha'}{\beta'}=\hat{G}_\ast(\hp{\alpha}{\beta})$ only involve terms of the form
$\alpha_m$ and $\beta_m$ for $m$  within some constant $k(\hat{G})$ of $n$ ($|m-n|<k$).
Thus, the action extends to $\HH_1(S \smallsetminus \Sigma, \Z)$. Furthermore,
it must be that $\hat{G}_\ast (\hom{y}) \cap \hat{G}_\ast (\hom{x})=\hom{y} \cap \hom{x}$
for all $\hom{x} \in H_1(S \smallsetminus \Sigma, \Z)$ and $y \in \HH_1(S \smallsetminus \Sigma, \Z)$.

\subsection{Invariant Planes}
Inside the action of $\Aut(S_{\geq 1})$ on the homological group $\HH_1(S_{\geq 1} \smallsetminus \Sigma, \R)$, there are embedded copies of 
the group actions on $\R^2$ given by the Veech groups $\Gamma(S_c)=\G^\pm_c$. 
Actually, the groups $\G_c$ make sense for all $c \in \R$ as representations
of the abstract group $\G^\pm$ into $\SL^\pm(2,\R)$. See theorem \ref{thm:veech_groups} 
for the definition of $\G_c$.
Given a $G \in \G^\pm$, we have the corresponding $\hat{G} \in \Aut(S_{\geq 1})$ as well
as the Veech group elements $G_c\in \G_c$ for all $c\in \R$.

\begin{theorem}
\label{thm:commutative_diagram}
For all $c \in \R$, the planes spanned by $h_c$ and $v_c$ are invariant. Furthermore, the linear map 
${\mathcal L}_c:\R^2 \to \HH_1(S_{\geq 1} \smallsetminus \Sigma, \R)$ defined in proposition \ref{prop:hol}
completes the commutative diagram
$$\begin{CD}
\R^2 @>G_c>> \R^2\\
@VV{\mathcal L}_cV @VV{\mathcal L}_cV\\
\HH_1(S_{\geq 1} \smallsetminus \Sigma, \R) @>\hat{G}_{\ast}>> \HH_1(S_{\geq 1} \smallsetminus \Sigma, \R)
\end{CD}$$
for all $G \in \G^\pm$.
\end{theorem}
\begin{proof}
By proposition \ref{prop:hol},
it follows that for all $\hom{x} \in H_1(S_{\geq 1} \smallsetminus \Sigma, \R)$, $(a,b) \in \R^2$, and $G \in \G^\pm$
we have
$$
{\mathcal L}_c(a,b) \cap \hat{G}_\ast^{-1}( \hom{x})
=(a,b) \wedge \hol_c \big(\hat{G}_\ast^{-1} (\hom{x})\big)
=(a,b) \wedge G_c^{-1} \big( \hol_c(\hom{x})\big)$$
By applying the invariance of $\wedge$ under $G_c$ and the invariance of $\cap$ under $\hat{G}_\ast$, we deduce
$$G_c(a,b) \wedge \hol_c(\hom{x}) = \hat{G}_\ast({\mathcal L}_c(a,b)) \cap \hom{x}.$$
Thus, by an application of proposition \ref{prop:hol}, to the left side we have
$${\mathcal L}_c( G_c(a,b)) \cap \hom{x}=\hat{G}_\ast({\mathcal L}_c(a, b)) \cap \hom{x}.$$
This equation says that the actions of $\hat{G}_\ast({\mathcal L}_c(a, b)) \cap$ and ${\mathcal L}_c( G_c(a,b)) \cap$
on homology are the same. Unfortunately, by the remarks before the definition of 
$\hom{z} \in \HH_1(S_{\geq 1} \smallsetminus \Sigma, \Z)$ in equation
\ref{eq:kernel}, this only implies that
$\hat{G}_\ast({\mathcal L}_c(a, b))={\mathcal L}_c( G_c(a,b))+s \hom{z}$ 
for some $s \in \R$.
It must therefore be explicitly checked that $s=0$ for the generators $G\in \{A,D,E\}$. 

For the readers sake, we only check the case of $G=A$. Let
$\hp{\alpha}{\beta}=\hat{A}_\ast({\mathcal L}_c(a, b))$ and $\hp{\alpha'}{\beta'}={\mathcal L}_c( A_c(a,b))$.
It is sufficient to check that $\alpha_1=\alpha'_1$. This is a straightforward calculation utilizing 
the $\hat{A}_\ast$ action given in 
equation \ref{eq:A-action}, the definition of $A_c$ given in theorem \ref{thm:veech_groups}, and
the definition of ${\mathcal L}_c$ given in proposition \ref{prop:hol}.
We compute $\alpha_1=-a-b=\alpha'_1$.
\end{proof}

We will now show that we can write every element of $\hom{x} \in H_1(S_{\geq 1} \smallsetminus \Sigma, \R)$ as an integral of 
vectors in the invariant planes ${\mathcal L}_{\cos \theta}$ for $\theta \in [-\pi, \pi]$. That is, we will find 
functions $r, s:[-\pi, \pi] \to \R$ so that
$$\hom{x}=
\int_{- \pi}^{\pi} {\mathcal L}_{\cos \theta} \big( r( \theta), s( \theta)\big)~d\theta.$$
We interpret such an integral {\em coordinatewise}. That is, we integrate each coordinate independently.
Then theorem \ref{thm:commutative_diagram} implies that for all $\hat{G} \in \Aut(S_{\geq 1})$, 
\begin{equation}
\label{eq:commutative_action}
\hat{G}_\ast( \hom{x})=
\int_{-\pi}^{\pi} {\mathcal L}_{\cos \theta} \Big(G_{cos \theta} \big(r( \theta), s( \theta)\big)\Big)~d\theta.
\end{equation}
Indeed, it turns out that there is a natural choice for $r$ and $s$ as polynomials in $c=\cos \theta$.

Let $\Poly$ denote the space of all polynomials. $\Poly^2$ will
denote the space of all pairs of polynomials.

\begin{theorem}
\label{thm:homology_as_integral}
Consider the linear embedding $\psi:H_1(S_{\geq 1} \smallsetminus \Sigma, \R) \to \Poly^2$ determined
by the following images of the basis elements, which are polynomials in $c=\cos \theta$.
$$\psi:\hom{\Cyl_{1,0}^n} \mapsto \big(  2 \cos((n-1) \theta)-2 \cos(n \theta), 0 \big)$$
$$\psi:\hom{\Cyl_{1,1}^n} \mapsto ((\cos((n-1) \theta)-\cos((n+1) \theta))\big( 1,1 \big)$$
This embedding satisfies the following for all $\hom{x} \in H_1(S_{\geq 1} \smallsetminus \Sigma, \R)$.
\begin{enumerate}
\item $\hom{x}=\frac{1}{4 \pi} \int_{- \pi}^{\pi} {\mathcal L}_{\cos \theta} \big( \psi(\hom{x})(\cos \theta)\big) ~d\theta$ coordinatewise.
\item $\psi(\hom{x})(1)=(0,0)$ and $\frac{d}{d\theta} \psi(\hom{x})(\cos \theta) |_{\theta=0}=(0,0)$.
\item $\hol_1(\hom{x})=\frac{d^2}{d\theta^2} \psi(\hom{x})( \cos \theta)|_{\theta=0}$.
\end{enumerate}
\end{theorem}

\begin{proof}
We first prove statement 1. It is enough to prove this for our basis of $H_1(S_{\geq 1} \smallsetminus \Sigma, \R)$.
We take the first case of $\hom{\Cyl_{1,0}^n}$.
Let 
$$\hp{\alpha}{\beta}=\frac{1}{4 \pi} \int_{- \pi}^{\pi} {\mathcal L}_{\cos \theta} (2 \cos((n-1) \theta)-2 \cos(n \theta))~d\theta$$
We wish to show $\hp{\alpha}{\beta}=\hom{\Cyl_{1,0}^n}$. Then 
$$\alpha_k=\frac{1}{4 \pi} \int_{- \pi}^{\pi} \big(2 \cos((n-1) \theta)-2 \cos(n \theta)\big) \sum_{j=-(n-1)}^{n-1} \cos j\theta ~d\theta.$$
Using ideas from Fourier series, 
it is not hard to see that $\alpha_k=1$ if $n=k$ and $\alpha_k=0$ otherwise. The 
$\beta_k$ must be zero, because ${\mathcal L}_{\cos \theta}(1,0)$ has zero $\beta$ portion, regardless of $\theta$.
Now we take the second case of $\hom{\Cyl_{1,1}^n}$. This time let 
$$\hp{\alpha}{\beta}=\frac{1}{4 \pi} \int_{- \pi}^{\pi} (\cos((n-1) \theta)-\cos((n+1) \theta)) {\mathcal L}_{\cos \theta} (1,1)~d\theta.$$
The $\alpha$ coordinates of ${\mathcal L}_{\cos \theta} (1,1)$ are all zero, so for
all $k$ we have $\alpha_k=0$. We also have
$$\beta_k=\frac{1}{4 \pi} \int_{- \pi}^{\pi} \big(\cos((n-1) \theta)-\cos((n+1) \theta)\big) 
{\textstyle \Big(\frac{2 \sum_{j=-(n-1)}^n \cos j\theta}{1+\cos \theta}\Big)}~d\theta.$$
Trigonometric tricks can be used to show that for all $n \geq 1$,
$$\frac{\cos((n-1) \theta)-\cos((n+1) \theta)}{1 + \cos \theta}=
2 \sum_{j=-(n-1)}^{n} (-1)^{j+n+1} \cos j\theta$$
Interestingly, the expression inside the integral reduces to a symmetric expression in $n$ and $k$.
$$\beta_k=\frac{1}{4 \pi} \int_{- \pi}^{\pi} 2 \cos ((n-k) \theta) - 2 \cos ((n+k) \theta)~d\theta.$$
So $\beta_k=1$ when $n=k$ and $\beta_k=0$ otherwise.

Statement $2$ is a trivial calculation, so we omit it. Statement $3$ is trivial as well.
$$\hol_1(\hom{\Cyl_{1,0}^n})=(4n-2,0)=\frac{d^2}{d\theta^2}(2 \cos((n-1) \theta)-2 \cos(n \theta),0)|_{\theta=0}$$
$$\hol_1(\hom{\Cyl_{1,1}^n})=(4n,4n)=\frac{d^2}{d\theta^2}((\cos((n-1) \theta)-\cos((n+1) \theta))\big( 1,1 \big)|_{\theta=0}$$
\end{proof}

\subsection{An asymptotic formula}

We say two sequences, $\langle r_n \rangle$ and $\langle s_n \rangle$, 
are {\em asymptotic} if $\lim_{n \to \infty} \frac{r_n}{s_n}=1$.
We denote this by $r_n \asymp s_n$.

\begin{lemma}[Multiplicative asymptotic growth] 
\label{lem:general asymptotics}
Let $\epsilon>0$, $a>0$, and $f,g:[-\epsilon,\epsilon] \to \R$ be $C^2$ functions satisfying 
\begin{enumerate}
\item $f(0)=\frac{d}{dx}f (0)=0$ and $\frac{d^2}{dx^2}f(0)=2$.
\item $g(0)=1$, $\frac{d}{dx}g(0)=0$, and $\frac{d^2}{dx^2}g(0)=-2a$.
\item $|g(x)|<1$ on $[-\epsilon,0) \cup (0,\epsilon]$.
\end{enumerate}
Then, the sequences given by
$$r_n=\int_{- \epsilon}^{\epsilon} f(x) g^n(x)~dx 
\qquad \textrm{and} \qquad
s_n=\frac{\sqrt{\pi}}{2 a^\frac{3}{2} n^\frac{3}{2}}$$
are asymptotic.
\end{lemma}

We prove this multiplicative asymptotic growth lemma in section \ref{ss:growth_lemma} of the appendix.

\subsection{Asymptotic growth of homology classes}
\label{sect:asymptotic_homology}

We will say a sequence of homology classes $\hom{x_m}=\hp{\alpha^m}{ \beta^m}$ 
in $\HH_1(S_{\geq 1} \smallsetminus \Sigma, \R)$
converges {\em coordinatewise} to 
$\hom{x_\infty}=\hp{\alpha^\infty}{\beta^\infty} \in \HH_1(S_{\geq 1} \smallsetminus \Sigma, \R)$
if for every $n$
$$\lim_{m \to \infty} \alpha^m_n = \alpha^\infty_n 
\quad \textrm{ and } \quad
\lim_{m \to \infty} \beta^m_n = \beta^\infty_n$$

An element $\hat{G} \in \Aut(S_{\geq 1})$ acts {\em hyperbolically} on $S_1$
if there is an eigenvalue $\lambda_1>1$ of $G_1={\bf D}(\hat{G}) \in \Gamma(S_1)$.
Let ${\bf v^+_1} \in \R^2$ be a unit length eigenvector with the eigenvalue $\lambda_1$,
and let ${\bf v^-_1} \in \R^2$ be a unit eigenvector with eigenvalue $\pm 1/\lambda_1$.
\begin{theorem}[Asymptotics of Homology]
\label{thm:asymptotics_of_homology}
There is a constant $\kappa_G> 0$ depending only on $\hat{G}$ so that
given any homology class $\hom{x} \in H_1(S_{\geq 1} \smallsetminus \Sigma, \R)$, 
the sequence of homology classes 
$$\hom{x_m}=\frac{m^{\frac{3}{2}} \hat{G}_\ast^m(\hom{x})}{\lambda_1^m} \in H_1(S_{\geq 1} \smallsetminus \Sigma, \R)$$
converge coordinatewise as $m \to \infty$ to 
$$\hom{x_\infty}=\kappa_G {\mathcal L_1}(\textit{proj}_{{\bf v^-_1},{\bf v^+_1}} \hol_1(\hom{x})),$$
where $\textit{proj}_{{\bf v^-_1},{\bf v^+_1}} \hol_1(\hom{x})$ 
denotes the projection of the holonomy vector in the direction 
${\bf v^-_1}$ onto the direction ${\bf v^+_1}$.
Furthermore, 
$$\kappa_G=\frac{1}{2 \sqrt{\pi}} \big(\frac{d}{dc} [2 \log \lambda_c] \Big|_{c=1} \big)^{\frac{-3}{2}}$$
where $\log$ is the natural logarithm and $\lambda_c$ denotes the greatest eigenvalue of $G_c$. Note,
$2 \log \lambda_c$ is geometrically the translation distance of the action of $G_c$ on the hyperbolic plane.
\end{theorem}

The first idea is to rewrite $\hom{x}=\hp{\alpha}{\beta}$ using theorem \ref{thm:homology_as_integral}.
There is a polynomial $\psi(\hom{x}):\R \to \R^2$ in the variable $c=\cos \theta$ such that
$$\hom{x}=\frac{1}{4 \pi} \int_{- \pi}^{\pi} {\mathcal L}_{\cos \theta} \big( \psi(\hom{x})(\cos \theta)\big) ~ d\theta.$$
Then theorem \ref{thm:commutative_diagram} and equation \ref{eq:commutative_action} tell us that we can write
\begin{equation}
\label{eq:homology_as_integral}
\hom{x_m}=
\frac{m^{\frac{3}{2}}}{4 \pi \lambda_1^m} \int_{- \pi}^{\pi} {\mathcal L}_{\cos \theta} \circ G_{\cos \theta}^m \big(\psi(\hom{x})(\cos \theta) \big) ~ d\theta.
\end{equation}
Now we will rewrite the action of $G_{\cos \theta}$ in a different basis depending on $\theta$. Let $\lambda_c$ denote the largest eigenvalue
of $G_{\cos \theta}$ and let ${\bf v_c^+}$ be a corresponding unit eigenvector, and ${\bf v_c^-}$ be a unit eigenvector with eigenvalue $\pm\lambda_c^{-1}$. These choices only make sense so long as $|\textit{Trace}(G_{\cos \theta})|>2$. The sign $\pm$ is invariant under $\theta$, since it determines whether $\hat{G}$ is orientation
preserving or reversing. Since $G_{\cos \theta}$ varies polynomially in $c=\cos \theta$,
by restricting to a small enough interval, $\delta\leq c \leq 1$, we may assume
all these quantities are $C^\infty$ in $c= \cos \theta$.
Then we can find $C^\infty$ functions ${\bf p},{\bf q}:\cos^{-1}([-\epsilon,\epsilon]) \to \R^2$
\begin{equation}
\label{eq:psi_decomp}
\psi(\hom{x})(\cos \theta)={\bf p}(\cos \theta)+{\bf q}(\cos \theta)
\end{equation}
with ${\bf p}(\cos \theta)$ parallel to the expanding direction ${\bf v_c^+}$ and
${\bf q}(\cos \theta)$ parallel to ${\bf v_c^-}$. Formulas for these functions 
exist in terms of the polynomials making up $G_{\cos \theta}$ and $\psi(\hom{x})(\cos \theta)$.
${\bf p}(\cos \theta)$ is the projection of $\psi(\hom{x})(\cos \theta)$ in the
${\bf v_c^-}$ direction onto the ${\bf v_c^+}$ direction, and ${\bf q}(\cos \theta)$
is the reverse.
By choosing $\epsilon$ sufficiently small, ${\bf p}$ and ${\bf q}$ may be made $C^\infty$
functions of $\cos \theta$ on their domains. This allows us to write the simplification
$$G_{\cos \theta}^m \big(\psi(\hom{x})(\cos \theta)\big)=\lambda_c^m {\bf p}(\cos \theta) + (\pm \lambda_c)^{-m} {\bf q}(\cos \theta).$$
In summary, we may break up equation \ref{eq:homology_as_integral} as a sum of three integrals
\begin{equation}
\label{eq:three_integrals}
\begin{array}{rl}
\hom{x_m} = \frac{m^{\frac{3}{2}}}{4 \pi \lambda_1^m} \Big( &
\int_{- \epsilon}^{\epsilon} \lambda_c^m {\mathcal L}_{\cos \theta} \big({\bf p}(\cos \theta)\big)~d \theta
+ \\
& \int_{- \epsilon}^{\epsilon} (\pm \lambda_c)^{-m} {\mathcal L}_{\cos \theta} \big({\bf q}(\cos \theta)\big)~d \theta
+ \\
& \int_{[- \pi, -\epsilon] \cup [\epsilon,\pi]} {\mathcal L}_{\cos \theta} \circ G_{\cos \theta}^m \big(\psi(\hom{x})(\cos \theta) \big) ~ d\theta
~ \Big)
\end{array}
\end{equation}
We will show that in the limit the first integral contributes, but the other two do not. The first integral contributes by
the multiplicative asymptotic growth lemma. The lemmas below are the keys to these facts.

\begin{lemma}[Derivative of the eigenvalue]
\label{lem:derivative_of_eig}
Let $\hat{G} \in \Aut(S_{\geq 1})$ be chosen so that $\hat{G}$ acts hyperbolically on $S_1$.
Let $\lambda_c$ be the largest eigenvalue of $G_c$. Then 
$$\frac{d}{dc} \lambda_c\Big|_{c=1}>0.$$
\end{lemma}

\begin{lemma}[Eigenvalue upper bound]
\label{lem:bounds}
Let $\hat{G} \in \Aut(S_{\geq 1})$ be chosen so that $\hat{G}$ acts hyperbolically on $S_1$.
Let $\lambda_c$ be the largest eigenvalue of $G_c$. Then 
$|\lambda_c| < \lambda_1$ for all $-1 \leq c <1$.
\end{lemma}

We will prove these lemmas in section \ref{sect:appendix_trace}.

\begin{proof}[Proof of theorem \ref{thm:asymptotics_of_homology}]
We continue the discussion below the statement of the theorem, where we left off with equation \ref{eq:three_integrals}.

We will compute the contribution of the first integral of equation \ref{eq:three_integrals}, 
$$I_1(m)=\frac{m^{\frac{3}{2}}}{4 \pi}
\int_{- \epsilon}^{\epsilon} 
(\frac{\lambda_c}{\lambda_1})^m {\mathcal L}_{\cos \theta} \big({\bf p}(\cos \theta)\big)~d \theta.$$
We will show that we can apply the asymptotic growth lemma.  Set
$$f(\theta)= {\mathcal L}_{\cos \theta} \big({\bf p}(\cos \theta)\big).$$
In order to apply the lemma, we must compute some derivatives. First we will examine ${\bf p}(\cos \theta)$.
We may write 
$${\bf p}(\cos \theta)=\big( \psi(\hom{x})(\cos \theta) \cdot {\bf v^+_c} \big) {\bf v^+_c}$$
Clearly ${\bf p}(1)=0$ since $\psi(\hom{x})(1)=0$, by the first part of statement 2 of theorem \ref{thm:homology_as_integral}. Also $\frac{d}{d \theta} {\bf p}(\cos \theta)|_{\theta=0}=(0,0)$ because ${\bf p}$ is a function
of $\cos \theta$. We will also need the second derivative. Using the product rule for derivatives
and applying theorem \ref{thm:homology_as_integral}, we can compute
$$\frac{d^2}{d \theta^2}{\bf p}(\cos \theta)\Big|_{\theta=0}=
\big( \frac{d^2}{d \theta^2} \psi(\hom{x})(\cos \theta) \cdot {\bf v^+_1} \big) {\bf v^+_1}
=\big( \hol_1(\hom{x}) \cdot {\bf v^+_1} \big) {\bf v^+_1}.$$
(The product rule gives several terms, but this is the only one that contributes, since
the zeroth and first derivative of $\psi(\hom{x})$ is zero.) Note, this is precisely
the quantity $\textit{proj}_{{\bf v^-_1},{\bf v^+_1}} \hol_1(\hom{x})$ from the statement of the theorem.
We return to analyzing $f(\theta)$. By proposition \ref{prop:hol}, each coordinate $i$ of $f(\theta)$ 
may be written in the form
$$f_i(\theta)=\rho_i(\cos \theta) L_i \big({\bf p}(\cos \theta)\big),$$
where $\rho_i(\cos \theta)$ is a polynomial in $\cos \theta$ and $L_i:\R^2 \to \R$ is a linear map
independent of $\cos \theta$. Since the zeroth and first derivatives of ${\bf p}$ vanish, 
so do the zeroth and first of $f_i$. Thus, the zeroth and first derivatives of $f(\theta)$ vanish.
Moreover, we can see
$$\frac{d^2}{d\theta^2} f_i(\theta)\Big|_{\theta=0}=
\rho_i(\cos \theta) L_i \big(\frac{d^2}{d\theta^2} {\bf p}(\cos \theta)\big)=
\rho_i(\cos \theta) L_i \big(\textit{proj}_{{\bf v^-_1},{\bf v^+_1}} \hol_1(\hom{x}))\big).$$
This implies that the following equation is true coordinatewise.
$$\frac{d^2}{d\theta^2} f(\theta)=
{\mathcal L}_{\cos \theta} \big(\frac{d^2}{d\theta^2}{\bf p}(\cos \theta)\big)=
{\mathcal L}_{\cos \theta} \big(\textit{proj}_{{\bf v^-_1},{\bf v^+_1}} \hol_1(\hom{x})\big).$$
Thus, the function $f(\theta)$ satisfies the first condition of the asymptotic growth lemma coordinatewise
up to a scalar multiple.

Now set $g(\theta)=\frac{\lambda_c}{\lambda_1}$. Clearly, $g(0)=1$. Also
$\frac{d}{d\theta}g(0)=0$, because $g$ is a function of $\cos \theta$.
For the same reason, we have
$$\frac{d^2}{d\theta^2} g(0)
=-\frac{d}{dc} (\frac{\lambda_c}{\lambda_1})\Big|_{c=1}
=\frac{-\frac{d}{dc} \lambda_c\big|_{c=1}}{\lambda_1}=-\frac{d}{dc} \log \lambda_c \Big|_{c=1}<0,$$
with the last inequality following from lemma \ref{lem:derivative_of_eig}.
Thus $g(\theta)$ satisfies the second condition of the asymptotic growth lemma.
Further, lemma \ref{lem:bounds} implies $g(\theta)$ satisfies the third condition
of the asymptotic growth lemma.

By the asymptotic growth lemma, 
$$
\begin{array}{rcl}
I_1(m) & = & \frac{m^{\frac{3}{2}}}{4 \pi}
\int_{- \epsilon}^{\epsilon} g(\theta)^m f(\theta)~d \theta
=\frac{\frac{1}{2}{\mathcal L}_{\cos \theta} \big(\textit{proj}_{{\bf v^-_1},{\bf v^+_1}} \hol_1(\hom{x})\big)}
{8 \sqrt{\pi} \big(\frac{1}{2}\frac{d}{dc} \log \lambda_c\big|_{c=1} \big)^{\frac{3}{2}}} \\
& = & \frac{{\mathcal L}_{\cos \theta} \big(\textit{proj}_{{\bf v^-_1},{\bf v^+_1}} \hol_1(\hom{x})\big)}
{2 \sqrt{\pi} \big(2\frac{d}{dc} \log \lambda_c\big|_{c=1} \big)^{\frac{3}{2}}}
=\kappa_H {\mathcal L}_{\cos \theta} \big(\textit{proj}_{{\bf v^-_1},{\bf v^+_1}} \hol_1(\hom{x})\big),
\end{array}
$$
which is the limit $\hom{x_\infty}$ claimed by the theorem.

The second integral from equation \ref{eq:three_integrals} may be written as
$$I_2(m)=\frac{1}{4 \pi}
\int_{- \epsilon}^{\epsilon} 
(\frac{1}{\pm \lambda_c \lambda_1})^m m^{\frac{3}{2}} {\mathcal L}_{\cos \theta} \big({\bf q}(\cos \theta)\big)~d \theta.$$
The function inside the integral goes to zero pointwise in every coordinate, since $(\frac{1}{\lambda_c \lambda_1})^m$, is exponentially decaying, while $m^{\frac{3}{2}}$ grows polynomially.
Thus, this term does not contribute to equation \ref{eq:three_integrals}.

Finally, we consider the third integral from equation \ref{eq:three_integrals}
$$I_3(m)=
\frac{1}{4 \pi}  \int_{[- \pi, -\epsilon] \cup [\epsilon,\pi]} 
{\mathcal L}_{\cos \theta} \Big( \frac{m^{\frac{3}{2}}}{\lambda_1^m} G_{\cos \theta}^m \big(\psi(\hom{x})(\cos \theta) \big)\Big) ~ d\theta.$$
The function inside this integral also goes to zero pointwise in every coordinate. By lemma \ref{lem:bounds},
for any $\theta$ in the domain of integration with $\textit{Trace}(G_{\cos \theta})$, the exponential growth
of the absolute values of entries in the sequence of matrices 
$G_{\cos \theta}^m$ must be slower than $\lambda_1$. Moreover,
by continuity of $G_{\cos \theta}$ and by compactness of the domain of integration, there must be
a $\lambda'<\lambda_1$ so that the exponential growth of entries of $G_{\cos \theta}^m$ must be slower
than by a factor of $\lambda'$.  The entries of $\frac{m^{\frac{3}{2}}}{\lambda_1^m} G_{\cos \theta}^m$
decay exponentially. Now, if $G_{\cos \theta}$ is parabolic then entries of $G_{\cos \theta}^m$ grow 
at most linearly.
Thus, $\frac{m^{\frac{3}{2}}}{\lambda_1^m} G_{\cos \theta}^m$ decays exponentially in this case too.
Finally, $G_{\cos \theta}$ may be elliptic. Then the entries of $G_{\cos \theta}^m$ are bounded close to zero.
In this case, $\frac{m^{\frac{3}{2}}}{\lambda_1^m} G_{\cos \theta}^m$ decays exponentially also.
\end{proof}

\subsection{Asymptotics of cylinder intersections}
\label{sect:asy_of_cyl_int}
By applying theorem \ref{thm:asymptotics_of_homology} to proposition \ref{prop:intersection_formula}, we obtain the 
following theorem.

\begin{theorem}[Asymptotics of cylinder intersections]
\label{thm:asymptotics_of_cylinder_intersections}
Let ${\mathcal A}$ and ${\mathcal B}$ be any two cylinders on $S_1$, then for any 
$\hat{G} \in \Aut(S_{\geq 1})$ acting hyperbolically on $S_1$, 
$$\lim_{m \to \infty} m^{\frac{3}{2}} \textit{Area}\big(\hat{G}^m({\mathcal A}) \cap {\mathcal B}\big)=
\kappa_G \textit{Area}({\mathcal A}) \textit{Area}({\mathcal B}),$$
with $\kappa_G$ as in theorem \ref{thm:asymptotics_of_homology}.
\end{theorem}
\begin{proof}
By proposition \ref{prop:intersection_formula}, the limit is equal to
$$
L=\lim_{m \to \infty}
m^{\frac{3}{2}}
\frac{\big|\hat{G}_\ast^m(\hom{\gamma_{\mathcal A}}) \cap \hom{\gamma_{\mathcal B})} \big|}
{\big|G_1^m(\hol(\gamma_{\mathcal A})) \wedge \hol(\gamma_{\mathcal B})\big|}
\textit{Area}({\mathcal A}) \textit{Area}({\mathcal B}).
$$
Let $\textit{proj}_{{\bf v_-},{\bf v_+}} \hol_1(\hom{\gamma_{\mathcal A}})$ 
denote the projection of the holonomy vector in the direction of the compressing direction of $G_1$ onto
the expanding direction of $G_1$.
By theorem \ref{thm:asymptotics_of_homology}, we may rewrite the numerator using the asymptotic form
of $\hat{G}_\ast^m(\hom{\gamma_{\mathcal A}}) \cap \hom{\gamma_{\mathcal B})}$. And the term in the denominator is
asymptotic to $\lambda^m \textit{proj}_{{\bf v_-},{\bf v_+}} \hol_1(\hom{x}) \wedge \hol_1(\gamma_{\mathcal B})$.
Thus, 
$$L=\lim_{m \to \infty}
\frac{\big|\lambda^m \kappa_G \big({\mathcal L_1}(\textit{proj}_{{\bf v_-},{\bf v_+}} \hol_1(\hom{\gamma_{\mathcal A}})) 
\cap \hom{\gamma_{\mathcal B})} \big) \big|}
{\big|\lambda^m \textit{proj}_{{\bf v_-},{\bf v_+}} \hol_1(\hom{\gamma_{\mathcal A}}) \wedge \hol_1(\gamma_{\mathcal B})\big|}
\textit{Area}({\mathcal A}) \textit{Area}({\mathcal B}).
$$
Finally, 
${\mathcal L_1}(\textit{proj}_{{\bf v_-},{\bf v_+}} \hol_1(\hom{\gamma_{\mathcal A}})) 
\cap \hom{\gamma_{\mathcal B}}=
\textit{proj}_{{\bf v_-},{\bf v_+}} \hol_1(\hom{\gamma_{\mathcal A}}) \wedge \hol_1(\gamma_{\mathcal B})$,
by proposition \ref{prop:hol}.
Furthermore, $\textit{proj}_{{\bf v_-},{\bf v_+}} \hol_1(\hom{\gamma_{\mathcal A}})$ points in an irrational direction
while by corollary \ref{cor:trichotomy}, $\hol_1(\gamma_{\mathcal B})$ points in a rational direction. 
Hence, the wedge of the pair must be non-zero. So, we can cancel and get the limit we claimed.
\end{proof}

We have the following corollary on the action of a hyperbolic element $\hat{G} \in \Aut(S_{\geq 1})$ on $S_1$.
\begin{corollary}
The action of $\hat{G}$ is not recurrent with respect to Lebesgue measure.
\end{corollary}
\begin{proof}
If the $\hat{G}$ action were recurrent with respect to Lebesgue measure, $\mu$, then every measurable set $A$ would
satisfy
$$\sum_{m \in \N} \mu(\hat{G}^m (A) \cap A)=\infty.$$
But for $A={\mathcal A}$ a cylinder, 
$\mu(\hat{G}^m ({\mathcal A}) \cap {\mathcal A}) \asymp \kappa_G m^{\frac{-3}{2}} \textit{Area}({\mathcal A})^2$.
And, 
$$\sum_{m \in \N} m^{\frac{-3}{2}}=\zeta(\frac{3}{2}) \approx 2.61238 < \infty.$$
So, by the theorem, the sum can not reach $\infty$.
\end{proof}

\section{Appendix}
\subsection{The multiplicative asymptotic growth lemma}
\label{ss:growth_lemma}
We will prove lemma \ref{lem:general asymptotics} essentially by example. Our principle example is the the famous 
Catalan numbers\footnote{The Catalan numbers are Sloane's sequence A000108. See \url{http://www.research.att.com/~njas/sequences/A000108}.}, 
which can be defined inductively by $c_0=1$ and 
\begin{equation}
\label{eq:catalan}
c_0=1
\quad \textrm{ and } \quad
c_{n+1}=(4-\frac{6}{n+2})c_n.
\end{equation}
The numbers $c_i$ are all integers. An alternate formula for them 
is $c_n=\frac{2n!}{n!^2(n+1)}$. It is well known that 
\begin{equation}
\label{eq:catalan_asymptotics}
c_n \asymp \frac{4^n}{\sqrt{\pi} n^{3/2}}.
\end{equation}
This can be derived from Stirling's formula, $n! \asymp \sqrt{2 \pi n} \frac{n^n}{e^n}$.
We will now give an interesting formula for the Catalan numbers, which may be new.

\begin{proposition}[Catalan numbers]
\label{prop:catalan}
The Catalan numbers are given by the formula
$$c_n=\frac{4^{n-1}}{\pi} \int_{-\pi}^{\pi} (2-2\cos \theta) \Big(\frac{1+\cos \theta}{2}\Big)^n ~d \theta$$
\end{proposition}
\begin{proof}
The proof is essentially induction by repeated integration by parts.
It is not hard to check the base case, that $c_0$ from the statement of the proposition is equal to one.

Let $x_n=\int_{-\pi}^{\pi} (2-2\cos \theta) \Big(\frac{1+\cos \theta}{2}\Big)^n ~d \theta$. Using the half-angle identities 
$\frac{1-\cos \theta}{2}=\sin^2 \frac{\theta}{2}$ and $\frac{1+\cos \theta}{2}=\cos^2 \frac{\theta}{2}$ and
by making the substitution $\alpha=\frac{\theta}{2}$ we see $x_n=8 y_n$ where
$$y_n=\int_{-\frac{\pi}{2}}^{\frac{\pi}{2}} \sin^2 \alpha \cos^{2n} \alpha ~d \alpha.$$
We now apply integration by parts with $u=\cos^{2n} \alpha$, $du=-2 n \sin \alpha \cos^{2n-1} \alpha$,
$v=\frac{\alpha-\sin \alpha \cos \alpha}{2}$, and $dv=\sin^2 \alpha$ to obtain
\begin{equation}
\label{eq:newintegral}
y_n=\frac{\pi}{2}+
\int_{-\frac{\pi}{2}}^{\frac{\pi}{2}} \alpha \sin \alpha \cos^{2n-1} \alpha ~d\alpha   
-n y_n
\end{equation}
Again we apply integration by parts. This time we let $u=\cos^{2n-1} \alpha$, 
$du=-(2 n-1) \sin \alpha \cos^{2n-2} \alpha$, $v=\sin \alpha-\alpha \cos \alpha$, and $dv=\alpha \sin \alpha$.
We see
$$y_n=\frac{\pi}{2}-n \pi-n(2n-1) \int_{-\frac{\pi}{2}}^{\frac{\pi}{2}} \alpha \sin \alpha \cos^{2n-1} \alpha ~d\alpha
+n(2n-1) y_{n-1} - n y_n$$
But, this new integral appeared also in equation \ref{eq:newintegral}, so we may substitute yielding
$$y_n=\frac{\pi}{2}-n \pi-n(2n-1) \Big( \frac{(n+1) y_n-\frac{\pi}{2}}{n}\Big)
+n(2n-1) y_{n-1} - n y_n$$
Then we solve for $y_n$ in terms of $y_{n-1}$ and see $y_n=\frac{2n-1}{2n+2}y_{n-1}$. Therefore,
$x_n=\frac{2n-1}{2n+2}x_{n-1}$ as well. And with $c_n$ as in the statement of the proposition, we see
$c_{n}=\frac{4^{n-1}}{\pi} x_{n}$, so that 
$$\frac{c_{n+1}}{c_n}=4 \frac{x_{n+1}}{x_n}=4 \frac{2n+1}{2n+4}=4-\frac{6}{2+n}$$
Thus, our numbers satisfy the recurrence relation of the Catalan numbers. See equation \ref{eq:catalan}.
\end{proof}

The following corollary is a special case of lemma \ref{lem:general asymptotics}, which proves it for one pair of functions
for each $a>0$.

\begin{corollary}
\label{cor:example_asymptotics}
For all constants $a>0$ let $f_a$ and $g_a$ be defined
$$f_a(\theta)= \frac{2-2\cos(2 \sqrt{a} \theta)}{4 a}
\quad \textrm{ and } \quad
g_a(\theta)=\frac{1+\cos (2 \sqrt{a} \theta)}{2}.
$$
Let $0<\epsilon< \frac{\pi}{\sqrt{a}}$. Then the sequence
given by $r_n=\int_{- \epsilon}^{\epsilon} f_a(\theta) g_a(\theta)^n ~d \theta$ is
asymptotic to $s_n=\frac{\sqrt{\pi}}{2 a^\frac{3}{2} n^\frac{3}{2}}$.
\end{corollary}
\begin{proof}
We apply the change of variables $x=2 \sqrt{a} \theta$. We see
$$r_n=\frac{1}{8 a^{\frac{3}{2}}}\int_{\frac{-\epsilon}{2 \sqrt{a}}}^{\frac{\epsilon}{2 \sqrt{a}}} 
(2 -2 \cos{x})\Big(\frac{1+\cos x}{2}\Big)^n ~d \theta.$$
Temporarily, let $h_n(x)=(2 -2 \cos{x})(\frac{1+\cos x}{2})^n $. Then, we can write 
\begin{equation}
\label{eq:r_n}
r_n=\frac{1}{8 a^{\frac{3}{2}}} 
\Big( \int_{-\pi}^{\pi} h_n(x)~dx - 
\int_{[-\pi,\frac{-\epsilon}{2 \sqrt{a}}] \cup [\frac{\epsilon}{2 \sqrt{a}},\pi]} h_n(x)~dx \Big)
\end{equation}
By proposition \ref{prop:catalan} and equation \ref{eq:catalan_asymptotics}, the sequence
\begin{equation}
\label{eq:s_n}
s_n=\frac{1}{8 a^{\frac{3}{2}}} \int_{-\pi}^{\pi} h_n(x)~dx=
\frac{\pi}{2 (4^n) a^{\frac{3}{2}}} c_n \asymp \frac{\sqrt{\pi}}{2   n^{\frac{3}{2}} a^{\frac{3}{2}}}.
\end{equation}
On the other hand, there is a constant $\delta<1$ so that $\frac{1+\cos x}{2} < \delta$ for all 
$x \in [-\pi,\frac{-\epsilon}{2 \sqrt{a}}] \cup [\frac{\epsilon}{2 \sqrt{a}},\pi]$. Therefore
$$\Big|\int_{[-\pi,\frac{-\epsilon}{2 \sqrt{a}}] \cup [\frac{\epsilon}{2 \sqrt{a}},\pi]} h_n(x)~dx \Big|
< \int_{[-\pi,\frac{-\epsilon}{2 \sqrt{a}}] \cup [\frac{\epsilon}{2 \sqrt{a}},\pi]} 4 \delta^n~dx
= 2 (\pi - \frac{\epsilon}{2 \sqrt{a}}) 4 \delta^n,$$
which is exponentially shrinking. Thus, equation \ref{eq:r_n} yields the asymptotic formula 
$$\lim_{n \to \infty} \frac{r_n}{(\frac{\sqrt{\pi}}{2   n^{\frac{3}{2}} a^{\frac{3}{2}}})}=
\lim_{n \to \infty} \frac{s_n}{(\frac{\sqrt{\pi}}{2   n^{\frac{3}{2}} a^{\frac{3}{2}}})}=1.$$
\end{proof}

\begin{proof}[Proof of Lemma \ref{lem:general asymptotics}]
Let $a$, $\epsilon$, $f$, $g$, $r_n$, and $s_n$ be as in the statement of the lemma.  
We will show that 
$$\lim_{n \to \infty} \frac{r_n}{n^{\frac{-3}{2}}}=\frac{\sqrt{\pi}}{2 a^{\frac{3}{2}}}$$
using the squeeze theorem. We will provide upper and lower bounds for this limit converging to 
the stated limit.

Take any $a_+>a$. Consider the functions $f_+(x)=\frac{1}{1+a_+-a} f_{a_+}(x)$ and $g_+(x)=g_{a_+}(x)$
with $f_{a_+}$ and $g_{a_+}$ as in corollary \ref{cor:example_asymptotics}. Then there is an 
$\epsilon_+$ with $0<\epsilon_+ \leq \epsilon$ so that
\begin{enumerate}
\item $g_+(x)<g(x)$ for all $x\neq 0$ with $x \in [-\epsilon_+,\epsilon+]$,
\item $f_+(x) < f(x)$ for all $x\neq 0$ with $x \in [-\epsilon_+,\epsilon+]$.
\end{enumerate}
Then it follows from corollary \ref{cor:example_asymptotics} that
$$\lim_{n \to \infty} \frac{r_n}{n^{\frac{-3}{2}}}
> \lim_{n \to \infty} \frac{\int_{-\epsilon_+}^{\epsilon_+} f_+(x) g_+(x)^n~dx}{n^{\frac{-3}{2}}}
= \frac{\sqrt{\pi}}{2 a_+^\frac{3}{2} (1+a_+-a)}.$$
Note that $a_+$ can be taken to be arbitrarily close to $a$ yielding the inequality
$$\lim_{n \to \infty} \frac{r_n}{n^{\frac{-3}{2}}} \geq \frac{\sqrt{\pi}}{2 a^\frac{3}{2}}.$$

A similar argument shows the opposite inequality.
\end{proof}
\subsection{Bounds and derivatives of eigenvalues}
\label{sect:appendix_trace}
The purpose of this section is to prove lemmas \ref{lem:bounds}
and \ref{lem:derivative_of_eig} of subsection \ref{sect:asymptotic_homology}.
Further, this section implicitly relates the constant $\kappa_G$ appearing in section \ref{sect:asymptotic_homology}
to billiards in hyperbolic triangles.

For a hyperbolic element $G$ of the hyperbolic isometry group $\PGL(2,\R)$, the greatest eigenvalue $\lambda$ 
has the geometric significance of 
\begin{equation}
\label{eq:dist}
\inf_{x \in \H^2} \textit{dist}(x, Gx)=2 \log \lambda.
\end{equation}
Moreover, this infimum is achieved. The collection of points where this infimum is achieved is a geodesic
in $\H^2$ called the {\em axis} of the hyperbolic isometry $G$. The axis has a canonical orientation
determined by the direction $G$ translates the geodesic. 
If $G$ belongs to a discrete group $\Gamma$,
then this axis projects to a curve of finite length in the quotient $\H^2/\Gamma$. 

In our case, $\Gamma$ will be a triangle group generated by three reflections. See figure \ref{fig:veechgroup}.
The quotients $\Delta=\H^2/\Gamma$ will be triangular billiard tables. The projection of the axis of an element
$G \in \Gamma$ will be a closed billiard path in the table $\H^2/\Gamma$. We can make this more concrete.
Given $G$ with eigenvalue $\lambda>1$, 
consider an orientation and distance preserving map 
$\tilde{\gamma}:[0,2 \log \lambda] \to \textit{Axis}(G)$. This curve projects to the closed billiard
path quotient $\gamma:[0,2 \log \lambda] \to \H^2/\Gamma$, which bounces
of edges according to the rules of optics. 

We will need to extend these ideas to triangular billiard tables which are not quotients of $\H^2$ by 
a discrete group. Let $\Delta$ be any triangle in $\H^2$ with edges marked by the set $\{a,b,c\}$. 
We will use $\ell_a$, $\ell_b$, and $\ell_c$ to denote the bi-infinite geodesics in $\H^2$ that contain
the marked edges. There is a natural representation to the isometry
group of the hyperbolic plane, $\rho_\Delta:\Z_2 \ast \Z_2 \ast \Z_2 \to \PGL(2,\R)$,
determined by sending the generators to the reflections $R_a$, $R_b$, and $R_c$
in the lines $\ell_a$, $\ell_b$, and $\ell_c$ respectively.
Given an element $G \in \Z_2 \ast \Z_2 \ast \Z_2$, it may be written as a product of reflections
$$G=R_{e_n} R_{e_{n-1}} \ldots R_{e_2} R_{e_1}.$$
We define the {\em orbit-class} $\Omega(G)$ to be the collection of all closed curves in $\H^2$ that visit the edges
$e_1, e_2, \ldots, e_n$ in that order. We define the {\em orbit-length} of $G$ to be
$\ell(G)=\inf_{\gamma \in \Omega(G)} \textit{length}(\gamma)$.

We have the following lemma.

\begin{lemma}
\label{lem:inf}
Let $\Delta$ and $G$ be as above. Let $\lambda$ denote the eigenvalue of $\rho_\Delta(G)$ with largest eigenvalue.
Then $2 \log |\lambda|\leq \ell(G)$.
\end{lemma}
\begin{proof}
The statement is certainly true unless $\lambda>1$. We will assume that there is a path $\gamma \in \Omega(G)$
with length less than $2 \log \lambda$ and draw a contradiction. We will construct a new path $\gamma'$ in
$\H^2$ with length equal to that of $\gamma$ such that $\rho_\Delta(G)$ translates the starting point of $\gamma'$
to the ending point of $\gamma'$. This will be a contradiction to equation \ref{eq:dist}.

We may assume $\gamma$ begins on the geodesic $\ell_{e_1}$ and then travels to $\ell_{e_2}$ and so on. 
Let $\gamma_i$
be the portion of $\gamma$ which travels from $\ell_{e_i}$ to $\ell_{e_{i+1}}$, with 
$\gamma_n$ the final arc of the path.
Let $W_i=R_{e_i} R_{e_{i-1}} \ldots R_{e_1}$ for $i=1 \ldots n$. 
We define $\gamma'$ to be the concatenation of arcs
$$\gamma'=W_1(\gamma_1) \cup W_2(\gamma_2) \cup \ldots \cup W_n(\gamma_n).$$
It can be seen by induction that this is a connected path. Since $\gamma$ was closed and began on $\ell_{e_1}$,
the the fixed point set of $W_1$, the transformation $G=W_n$ takes the starting point of $\gamma'$ to the ending
point of $\gamma'$ as planned. And, $\gamma$ has the same length as $\gamma'$.
\end{proof}

We will need to introduce one more idea. The {\em Klein model} for the hyperbolic plane consists of 
$\KH^2=\{(x^2,y^2) \in \R^2 ~|~x^2+y^2<1\}$. The boundary 
$\del \KH^2=\{(x^2,y^2) \in \R^2 ~|~x^2+y^2=1\}$. Geodesics in the Klein model
are Euclidean line segments.
Distance between points in the Klein model may be computed
in two ways. Let $P_1$ and $P_2$ be two points in $\KH^2$, and let $\overline{P_1 P_2}$ be the Euclidean
line through them. Let $Q_1$ and $Q_2$ be the two points of $\del \KH^2 \cap \overline{P_1 P_2}$, chosen so that
$P_1$ is closest in the Euclidean metric to $Q_1$. Then the distance between $P_1$ and $P_2$ is given
by the logarithm of the cross ratio,
\begin{equation}
\label{eq:cr}
\textit{dist}_{\KH^2} (P_1,P_2)=\log \Big(\frac{\textit{dist}_{\R^2}(P_1,P_2) \textit{dist}_{\R^2}(Q_1,Q_2)}
{\textit{dist}_{\R^2}(P_1,Q_1) \textit{dist}_{\R^2}(P_2,Q_2)}\Big).
\end{equation}
Alternately, we can compute distance by using the metric tensor $ds$. 
\begin{equation}
\label{eq:metric_tensor}
ds=\sqrt{\frac{dx^2+dy^2}{1-x^2-y^2}+\frac{(x dx+y dy)^2}{(1-x^2-y^2)^2}}.
\end{equation}
The distance between two points can be computed by integrating the metric tensor over the geodesic path
between them. See \cite{CFKP} or any hyperbolic geometry text for more details.

As discussed below theorem \ref{thm:veech_groups}, the groups $\Gamma(S_c)$ are generated by the reflections 
$A_c$, $B_c$, and $C_c$ in the sides of a triangle $\Delta_c$ in $\H^2$ together with $-I$, 
which acts trivially on $\H^2$. Recall figure \ref{fig:veechgroup}.
When $c=\cos \theta \leq 1$, the triangle has two ideal vertices and 
one vertex with angle $\theta$. This is true for all $c \leq 1$, not just the $c$ of the form $\cos \frac{2\pi}{n}$. 
For our purposes, we will think of $\Delta_c \subset \KH^2$. We define
\begin{equation}
\Delta_c=\textit{Convex Hull}(\{P_1,P_2,P_3\}) \subset  \KH^2
\end{equation}
where $P_1=(-1,0)$, $P_2=(-1,0)$, and $P_3=(0,\frac{\sqrt{1+c}}{\sqrt{2}})$. 
This triangle is isometric to the triangle described for the group $\Gamma(S_c)$. The reflection lines of the elements 
$A_c$, $B_c$, and $C_c$ are given
by $\overline{P_3 P_1}$, $\overline{P_1 P_2}$, and $\overline{P_2 P_3}$ respectively.

\begin{proof}[Proof of the eigenvalue upper bound lemma \ref{lem:bounds}]
Choose $\hat{G} \in \Aut(S_{\geq 1})$ be chosen so that $\hat{G}$ acts hyperbolically on $S_1$.
Then $G_1$ is a hyperbolic transformation of $\H^2$. The projection of the axis of $G_1$ to
$\H^2/\Gamma(S_1)$ minimizes length in the {\em orbit-class} $\Omega(G)$.
In particular, this billiard path $\gamma_1$ realizes the infimum discussed in lemma \ref{lem:inf}.

Let $\lambda_c$ denote the eigenvalue with greatest absolute value of $G_c \in \Gamma(S_c)$. 
We will use lemma \ref{lem:inf} to show that $|\lambda_c|<\lambda_1$ for all $-1 \leq c < 1$. 
We define the linear map 
\begin{equation}
\label{eq:mc}
M_c~:~\Delta_1 \to \Delta_c~:~(x,y)\mapsto(x,\frac{y\sqrt{1+c}}{\sqrt{2}}).
\end{equation}
We claim that $M_c$ shortens every line segment except line segments contained in the side
$\overline{P_1 P_2}$. Consider a segment $\overline{XY}$ with finite length in $\Delta_1$. Let $\overline{PQ}$ be
the geodesic containing $\overline{XY}$, so that $P,Q \in \del \KH^2$ with $P$ the closest to $X$ as
in the left side of figure \ref{fig:trimap}. Then by equation
\ref{eq:cr},
$$\begin{array}{rcl}
\textit{dist}_{\KH^2} (X,Y)& =& {\displaystyle \log \Big(\frac{\textit{dist}_{\R^2}(X,Y) \textit{dist}_{\R^2}(P,Q)}
{\textit{dist}_{\R^2}(X,P) \textit{dist}_{\R^2}(Y,Q)}\Big)} \\
& = &{\displaystyle  \log \Big(\frac{\textit{dist}_{\R^2}(M_c(X),M_c(Y)) \textit{dist}_{\R^2}(M_c(P),M_c(Q))}
{\textit{dist}_{\R^2}(M_c(X),M_c(P)) \textit{dist}_{\R^2}(M_c(Y),M_c(Q))}\Big).}
\end{array}
$$
And, let $P',Q' \in \del \KH^2$ be the points where the geodesic $\overline{M_c(P) M_c(Q)}$ intersects 
the boundary. Then,
$$\textit{dist}_{\KH^2} (M_c(X),M_c(Y))=\log \Big(\frac{\textit{dist}_{\R^2}(M_c(X),M_c(Y)) \textit{dist}_{\R^2}(P',Q')}
{\textit{dist}_{\R^2}(M_c(X),P') \textit{dist}_{\R^2}(M_c(Y),Q')}\Big).$$
It is a standard computation that $\textit{dist}_{\KH^2} (X,Y)>\textit{dist}_{\KH^2} (M_c(X),M_c(Y))$
so long as either $M_c(P) \neq P'$ or $M_c(Q) \neq Q'$. In particular, the only time this inequality 
could be false is when $PQ \subset \{(x,y)|y=0\}$. Our claim is proved.

\begin{figure}[h]
\begin{center}
\includegraphics[width=3.5in]{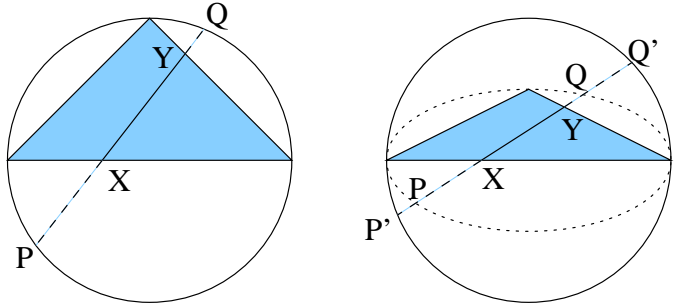}
\caption{A segment and its image under $M_c$. }
\label{fig:trimap}
\end{center}
\end{figure}

No finite length billiard path can have a segment contained in the line $y=0$, therefore 
$$\textit{length}\big( M_c(\gamma_1)\big) < \textit{length}\big( \gamma_1 \big).$$
whenever $-1 \leq c < 1$. Then by lemma \ref{lem:inf}, for all such $c$, 
$$2 \log |\lambda_c| \leq \textit{length}\big( M_c(\gamma_1)\big) < \textit{length}\big( \gamma_1 \big) = 2 \log \lambda_1.$$
Thus $|\lambda_c|<\lambda_1$. 
\end{proof}

\begin{proof}[Proof of the derivative of the eigenvalue lemma \ref{lem:derivative_of_eig}]
Let $\hat{G} \in \Aut(S_{\geq 1})$ be chosen so that $\hat{G}$ acts hyperbolically on $S_1$.
Let $\lambda_c$ be the largest eigenvalue of $G_c$. We must show
$$\frac{d}{dc} \lambda_c\Big|_{c=1}>0.$$
Clearly $\frac{d}{dc} \lambda_c$ exists, since entries of $G_c$ vary polynomially in $c$. We can afford
to restrict attention to the case of $c<1$, so we may recall the map $M_c: \Delta_1 \to \Delta_c$ of
equation \ref{eq:mc}.

Let $m=\frac{\sqrt{1+c}}{\sqrt{2}}$. We have $\frac{d}{dc} m=\frac{1}{2 \sqrt{2+2c}}$,
and $\frac{d}{dc} m|_{c=1}=\frac{1}{4}$.
By equation \ref{eq:metric_tensor},
the $\KH^2$ length of the vector ${\bf i}=(1,0)$ at the point $(x,y) \in \KH$ is given
by 
$$I_1=\frac{\sqrt{1-y^2}}{1-x^2-y^2}.$$
The length of the vector $M_c({\bf i})=(1,0)$ at the point $M_c(x,y)=(x,my)$ is given by
$$I_2=\frac{\sqrt{1-m^2 y^2}}{1-x^2-m^2y^2}.$$
We compute 
$$\frac{d}{d c} [\frac{I_2}{I_1}]_{c=1}=\frac{y^2(1+x^2-y^2)}{4(1-y^2)(1-x^2-y^2)}>0.$$
Note, we are perturbing $c$ in the negative direction. This says that off the line $y=0$, 
$M_c$ is compressing every horizontal vector enough to be detected
by the first derivative.
Let $J_1$ be the $\KH^2$ length of the vector ${\bf j}=(0,1)$ at the point $(x,y) \in \KH$ and
$J_2$ be the $\KH^2$ length of the vector $M_c({\bf j})=(0,m)$ at the point $M_c(x,y)=(x,my)$.
We have
$$J_1=\frac{\sqrt{1-x^2}}{1-x^2-y^2}
\quad \textrm{and} \quad
J_2=\frac{m\sqrt{1-x^2}}{1-x^2-m^2 y^2}.$$
We compute 
$$\frac{d}{d c} [\frac{J_2}{J_1}]_{c=1}=\frac{1-x^2+y^2}{4(1-x^2-y^2)}>0.$$
In this case, $M_c$ is compressing every vertical vector enough to be detected
by the first derivative.

The argument concludes in the same manner as the previous proof. Let $\gamma_1$ be the 
billiard path on $\Delta_1$ corresponding to $G_1$. The argument above tells us that
$\frac{d}{dc} \textit{length}(M_c(\gamma_1))=k>0.$ But for $c<1$, 
$$2 \log \lambda_c \leq \textit{length}(M_c(\gamma_1))=\textit{length}(\gamma_1)-k(1-c)+\textit{higher order terms}.$$
By taking a straight forward derivative, we get 
$$2 \log \lambda_c=2 \log \lambda_1-(1-c)\frac{2 \frac{d}{dc} \lambda_c |_{c=1}}{\lambda_1}+\textit{higher order terms}.$$
This pair of equations imply
$$\frac{d}{dc} \lambda_c |_{c=1} \geq \frac{k \lambda_1}{2}>0.$$
\end{proof}

\bibliographystyle{amsalpha}
\bibliography{../bibliography}
\end{document}